\documentclass[12pt,intlimits]{article}
\usepackage[a4paper,top=2cm,bottom=2cm,left=2cm,right=2cm,bindingoffset=5mm]{geometry}
\usepackage{amsmath}
\usepackage{amsthm}
\usepackage{amssymb}
\usepackage{graphicx}
\usepackage{epsfig}
\usepackage{hyperref}
\hypersetup{colorlinks=true,linkcolor=blue}
\newtheorem{theo}{Theorem}[section]
\newtheorem{lem}[theo]{Lemma}

\newtheorem{prop}[theo]{Proposition}

\theoremstyle{definition}
\newtheorem{defi}[theo]{Definition}
\theoremstyle{remark}
\newtheorem{remark}[theo]{Remark}

\newcommand{\be}{\begin{eqnarray}}
\newcommand{\ee}{\end{eqnarray}}
\newcommand{\bes}{\begin{eqnarray*}}
\newcommand{\ees}{\end{eqnarray*}}
\newcommand{\bi}{\begin{itemize}}
\newcommand{\ei}{\end{itemize}}
\newcommand{\ben}{\begin{enumerate}}
\newcommand{\een}{\end{enumerate}}

\newcommand{\La}{\mathcal{L}}
\newcommand{\D}{\mathcal{D}}

\newcommand{\B}{\mathcal{B}^{s,t}_\Omega}
\newcommand{\Bn}{\mathcal{B}^{s,t}_{\Omega_n}}
\newcommand{\BT}{\mathcal{B}^{s,t}_{\mathcal{T}}}

\newcommand{\Hcal}{\mathcal{H}}
\newcommand{\R}{\mathbb{R}}
\newcommand{\N}{\mathbb{N}}
\newcommand{\C}{\mathbb{C}}

\newcommand{\G}{\mathcal{G}}
\newcommand{\T}{\mathcal{T}}
\newcommand{\Ncal}{\mathcal{N}_{2-2s}^K}

\newcommand{\de}{\mathrm {d}}
\def\Ext{{\hbox{\rm Ext}}}
\def\einschr{\hbox{\kern1pt\vrule height 6pt\vrule  width6pt height 0.4pt depth0pt\kern1pt}}

\newcommand{\Lm}{\mathfrak{L}}

\DeclareMathOperator{\supp}{supp}

\newcommand{\uu}{u|_{\partial \Omega}}
\newcommand{\vv}{v|_{\partial \Omega}}

\title{\bf Dynamic boundary conditions for time dependent fractional operators on extension domains  }

\date{}

\begin{document}
\maketitle

\centerline{\scshape Simone Creo and Maria Rosaria Lancia}
\medskip
{\footnotesize

 \centerline{Dipartimento di Scienze di Base e Applicate per l'Ingegneria, Sapienza Universit\`{a} di Roma,
}
   \centerline{via Antonio Scarpa 16, 00161 Roma, Italy.}
   \centerline{E-mail: simone.creo@uniroma1.it,\quad maria.lancia@sbai.uniroma1.it}
}

\vspace{1cm}

\begin{abstract}
\noindent We consider a parabolic semilinear non-autonomous problem $(\tilde P)$ for a fractional time dependent operator $\mathcal{B}^{s,t}_\Omega$ with Wentzell-type boundary conditions in a possibly non-smooth domain $\Omega\subset\mathbb{R}^N$. We prove existence and uniqueness of the mild solution of the associated semilinear abstract Cauchy problem $(P)$ via an evolution family $U(t,\tau)$. We then prove that the mild solution of the abstract problem $(P)$ actually solves problem $(\tilde P)$ via a generalized fractional Green formula.
\end{abstract}

\medskip



\bigskip

\section*{Introduction}
\setcounter{equation}{0}

In this paper we consider a parabolic semilinear boundary value problem with dynamic boundary conditions for a generalized time dependent fractional operator in an extension domain $\Omega\subset\R^N$ having as boundary a $d$-set (we refer the reader to Section \ref{geometria} for the definitions). Problems of this type are also known as Wentzell-type problems. The problem is formally stated as follows:
\begin{equation}\notag
(\tilde P)\begin{cases}
\frac{\partial u}{\partial t}(t,x)+\B u(t,x)=J(u(t,x)) &\text{in $[0,T]\times\Omega$,}\\[2mm]
\frac{\partial u}{\partial t}(t,x)+C_s\Ncal u(t,x)+b(t,x)u(t,x)+\Theta^t_\alpha (u(t,x))=J(u(t,x))\, &\text{on $[0,T]\times\partial\Omega$},\\[2mm]
u(0,x)=\phi(x) &\text{in $\overline\Omega$},
\end{cases}
\end{equation}
where $0<s<1$, $\alpha$ is defined in \eqref{definizione alpha}, $\B$ and $\Theta^t_\alpha$ denote generalized time dependent fractional operators on $\Omega$ and $\partial\Omega$ (see \eqref{fracreglap} and \eqref{nonlocal-op.}) respectively, $T$ is a fixed positive number, $b$ is a suitable function depending also on $t$ which satisfies hypotheses \eqref{ipotesi b}, $\Ncal$ is the fractional conormal derivative defined in Theorem \ref{greenf}, $\phi$ is a given datum in a suitable functional space and $J$ is a mapping from $L^{2p}(\Omega,m)$ to $L^2(\Omega,m)$, for $p>1$, locally Lipschitz on bounded sets in $L^{2p}(\Omega,m)$ (see condition \eqref{LIPJ}), where $m$ is the measure defined in \eqref{defmisura}. We remark that $\B$ is a time dependent generalization of the regional fractional Laplacian $(-\Delta)^s_\Omega$ and $\Theta^t_\alpha$ plays the role of a regional fractional Laplacian of order $\alpha\in (0,1)$ on $\partial\Omega$ (see Section \ref{sezgreen}).

We approach this problem by proving that there exists a unique evolution family associated with the non-autonomous energy form $E[t,u]$ defined in \eqref{frattale}. More precisely, after introducing the energy form $E[t,u]$, we consider the following abstract Cauchy problem $(P)$ (see also \eqref{eq:5.1}):
\begin{equation}\notag%\label{eq:5.1}
(P)\begin{cases}
\frac{\partial u(t)}{\partial t}=A(t)u(t)+J(u(t))\quad\text{for $t\in[0,T]$},\\
u(0)=\phi,
\end{cases}
\end{equation}
where $A(t)\colon D(A(t))\subset L^2(\Omega,m)\to L^2(\Omega,m)$ is the family of operators associated to $E[t,u]$.

Crucial tools for proving existence and uniqueness of the (mild) solution of the non-autonomous abstract Cauchy problem $(P)$ are a fractional version of the Nash inequality on $L^2(\Omega,m)$, which in turn allows us to prove the ultracontractivity of the evolution family $U(t,\tau)$ (see Theorem \ref{ultracontr}), and a contraction argument in suitable Banach spaces. A generalized fractional Green formula, proved in Section \ref{sezgreen}, then allows us to deduce that the mild solution of problem $(P)$ actually solves problem $(\tilde P)$ in a suitable weak sense, see Theorem \ref{esistfrattale}.

The literature on boundary value problems with dynamic boundary conditions in smooth domains is huge: we refer to \cite{AP-NAsurv,F-G-G-Ro,goldstein} and the references listed in. On the contrary, the study of Wentzell problems in extension domains (in particular with fractal boundaries and/or interfaces) is more recent; among the others, we refer to \cite{La-Ve3,JEE,CLNpar,JEEfraz,creoZAA}.

The study of autonomous semilinear problems in extension domains with Wentzell-type boundary conditions for fractional operators is a rather recent topic. We refer to \cite{JEEfraz} for the linear case and to \cite{CLVmosco,CLNODEA,CLVtrasm} for the case $p\geq 2$. The literature on fractional operators is huge since they mathematically describe the so-called anomalous diffusion. This topic appears also in finance and probability. We refer to the papers \cite{Ab-Th,Ja,MeMiMo,schneider,CSF,valdinoci,mandelbrot}, which deal with models describing such diffusion.

On the other side, to consider the corresponding non-autonomous problems allows to tackle more realistic problems, and it is indeed a challenging task. To our knowledge, the first results on non-autonomous semilinear Wentzell problems for the Laplace operator in irregular domains are contained in \cite{LVnonaut}.

When investigating semilinear problems, both autonomous and non-autonomous, the functional setting is given by an interpolation space between the domain of the generator $A(t)$ and $L^2(\Omega,m)$ or the domain of a fractional power of $A(t)$. In the case of extension domains, possibly with fractal boundary, the domain of $A(t)$ is unknown. Our aim here is to extend to the fractional non-autonomous case the ideas and methods of \cite{LVnonaut} and \cite{Daners} under suitable hypotheses on $J(u)$. In order to use a fixed point argument in Banach spaces, a crucial tool is to prove suitable mapping properties for $J(u)$, which in turn deeply rely on the ultracontractivity of the evolution family. We stress the fact that the techniques used in \cite{LVnonaut} to prove the ultracontractivity property cannot be applied to the present case, since the non-autonomous form $E[t,u]$ is nonlocal. Here the ultracontractivity property is obtained by an abstract argument which deeply relies on a fractional Nash inequality on $L^2(\Omega,m)$.
When giving the strong interpretation of problem $(P)$, this functional setting allows us to prove that the unknown $u$ satisfies a dynamic boundary condition on $\partial\Omega$, whereas it was not possible to achieve it in \cite{LVnonaut}, due to the presence of the \lq\lq fractal Laplacian" on the boundary. %Moreover, The natural functional setting to consider problems with dynamical boundary conditions is the space $L^p(\Omega,m)$.

\medskip

The paper is organized as follows.
In Section \ref{preliminari} we introduce the geometry and the functional setting and we recall important general results, such that trace theorems, Sobolev-type embeddings for extension domains and Nash inequality (see Proposition \ref{Nash}).\\
In Section \ref{sezgreen} we introduce the time dependent operator $\B$ which governs the diffusion in the bulk and we introduce the notion of fractional conormal derivative $\Ncal$ via a generalized fractional Green formula (see Theorem \ref{greenf}).\\
In Section \ref{sec3} we introduce the nonlocal operator $\Theta^t_\alpha$ acting on $\partial\Omega$ and the non-autonomous energy form $E[t,u]$, we prove its properties and that there exists a unique evolution family $U(t,\tau)$ associated to $E[t,u]$.\\
In Section \ref{sec4} we prove some regularity properties of the evolution family, in particular its ultracontractivity (see Theorem \ref{ultracontr}).\\
In Section \ref{sec5} we consider the abstract Cauchy problem $(P)$ and we prove that it admits a unique local (mild) solution. We then prove that the unique solution is also global in time under suitable assumptions on the initial datum. Finally, we prove that the unique mild solution of $(P)$ solves in a suitable weak sense problem $(\tilde P)$.

\section{Preliminaries}\label{preliminari}
\setcounter{equation}{0}

\subsection{Functional spaces}\label{spazi funzionali}

Let $\G$ (resp. $\mathcal{S}$) be an open (resp. closed) set of $\R^N$.
By $L^p(\G)$, for $p\geq 1$, we denote the Lebesgue space with respect to the Lebesgue measure $\de\La_N$, which will be left to the context whenever that does not create ambiguity. By $L^p(\partial\G)$ we denote the Lebesgue space on $\partial\G$ with respect to a Hausdorff measure $\mu$ supported on $\partial \G$. By $\D(\G)$ we denote the space of infinitely differentiable functions with compact support on $\G$. By $C(\mathcal{S})$ we denote the space of continuous functions on $\mathcal{S}$ and by $C^{0,\vartheta}(\mathcal{S})$ we denote the space of H\"older continuous functions on $\mathcal{S}$ of order $0<\vartheta<1$.\\
By $H^s(\G)$, where $0<s<1$, we denote the fractional Sobolev space of exponent $s$. Endowed with the norm
\begin{equation*}
\|u\|^2_{H^s(\G)}=\|u\|^2_{L^2(\G)}+\iint_{\G\times\G} \frac{(u(x)-u(y))^2}{|x-y|^{N+2s}}\,\de\La_N(x)\de\La_N(y),
\end{equation*}
it becomes a Banach space. We denote by $|u|_{H^s(\G)}$ the seminorm associated to $\|u\|_{H^s(\G)}$ and by $(u,v)_{H^s(\G)}$ the scalar product induced by the $H^s$-norm. Moreover, we set
\begin{equation}\notag
(u,v)_s:=\iint_{\G\times\G}\frac{(u(x)-u(y))(v(x)-v(y))}{|x-y|^{N+2s}}\,\de\La_N(x)\de\La_N(y).
\end{equation}

In the following we will denote by $|A|$ the Lebesgue measure of a subset $A\subset\R^N$.
For $f\in H^{s}(\G)$, we define the trace operator $\gamma_0$ as
\begin{equation}\notag %\label{e3.33}
\gamma_0f(x):=\lim_{r\to 0}{1\over|B(x,r)\cap\G|}\int_{B(x,r)\cap\G}f(y)\,\de\La_N(y)
\end{equation}
at every point $x\in \overline{\G}$ where the limit exists. The above limit exists at quasi every $x\in \overline{\G}$ with respect to the $(s,2)$-capacity (see Definition 2.2.4 and Theorem 6.2.1 page 159 in \cite{AdHei}). From now on, we denote the trace operator simply by $f|_{\G}$; sometimes we will omit the trace symbol and the interpretation will be left to the context.
Moreover, we denote by $\Lm(X\to Y)$ the space of linear and continuous operators from a Banach space $X$ to a Banach space $Y$. If $X=Y$, we simply denote this space by $\Lm(X)$.

\noindent Throughout the paper, $C$ denotes possibly different constants. We give the dependence of constants on some parameters in parentheses. 

\subsection{$(\varepsilon,\delta)$ domains and trace theorems}\label{geometria}

We recall the definition of $(\varepsilon,\delta)$ domains. For details see \cite{Jones}.
\begin{definition}\label{defepsdelta} Let $\mathcal{F}\subset\R^N$ be open and connected. For $x\in\mathcal{F}$, let $\displaystyle d(x):=\inf_{y\in\mathcal{F}^c}|x-y|$. We say that $\mathcal{F}$ is an $(\varepsilon,\delta)$ domain if, whenever $x,y\in\mathcal{F}$ with $|x-y|<\delta$, there exists a rectifiable arc $\gamma\in\mathcal{F}$ of length $\ell(\gamma)$ joining $x$ to $y$ such that
\begin{center}
$\displaystyle\ell(\gamma)\leq\frac{1}{\varepsilon}|x-y|\quad$ and\quad $\displaystyle d(z)\geq\frac{\varepsilon|x-z||y-z|}{|x-y|}$ for every $z\in\gamma$.
\end{center}
\end{definition}

\medskip

\noindent We now recall the definition of $d$-set, referring to \cite{JoWa} for a complete discussion.
\begin{definition}\label{dset}
A closed nonempty set $\mathcal{S}\subset\R^N$ is a $d$-set (for $0<d\leq N$) if there exist a Borel measure $\mu$ with $\supp\mu=\mathcal{S}$ and two positive constants $c_1$ and $c_2$ such that
\begin{equation}\label{defindset}
c_1r^{d}\leq \mu(B(x,r)\cap\mathcal{S})\leq c_2 r^{d}\quad\text{for every }x \in\mathcal{S}.%0<r\leq 1.
\end{equation}
The measure $\mu$ is called a $d$-measure.
\end{definition}

\medskip

\noindent In this paper, we consider two particular classes of $(\varepsilon,\delta)$ domains $\Omega\subset\R^N$. More precisely, $\Omega$ can be a $(\varepsilon,\delta)$ domain having as boundary either a $d$-set or an arbitrary closed set in the sense of \cite{jonsson91}. For the sake of simplicity, from now on we restrict ourselves to the case in which $\partial\Omega$ is a $d$-set.

%In the following, $\mu(\partial\Omega)$ denotes the Hausdorff measure of $\partial\Omega$.

We suppose that $\Omega$ can be approximated by a sequence $\{\Omega_n\}$ of domains such that, for every $n\in\N$,
\begin{equation*}
(\Hcal)\begin{cases}
\Omega_n\text{ is bounded and Lipschitz;}\\[2mm]
\Omega_n\subseteq \Omega_{n+1};\\[2mm]
\Omega=\displaystyle\bigcup_{n=1}^\infty \Omega_n.
\end{cases}
\end{equation*}

\noindent The reader is referred to \cite{JEEfraz} and \cite{CLVmosco} for examples of such domains.

%In order for the trace to be well defined, from now on we suppose that $$\frac{1}{2}<s<1.$$

\bigskip

We recall the definition of Besov space specialized to our case. For generalities on Besov spaces, we refer to \cite{JoWa}.
\begin{definition}
Let $\mathcal{F}$ be a $d$-set with respect to a $d$-measure $\mu$ and $0<\alpha<1$. ${B^{2,2}_\alpha(\mathcal{F})}$ is the space of functions for which the following norm is finite,
$$
\|u\|^2_{B^{2,2}_\alpha(\mathcal{F})}=\|u\|^2_{L^2(\mathcal{F})}+\iint_{|x-y|<1}\frac{|u(x)-u(y)|^2}{|x-y|^{d+2\alpha}}\,\de\mu(x)\,\de\mu(y).
$$
\end{definition}

\noindent In the following, we will denote the dual of the Besov space $B^{2,2}_\alpha(\mathcal{F})$ with $(B^{2,2}_\alpha(\mathcal{F}))'$; we point out that this space coincides with the space $B^{2,2}_{-\alpha}(\mathcal{F})$ (see \cite{JoWa2}).

From now on, let 
\begin{equation}\label{definizione alpha}
\alpha:=s-\frac{N-d}{2}\in (0,1).
\end{equation}
We now state the trace theorem for functions in $H^s(\Omega)$, where $\Omega$ is a bounded $(\varepsilon,\delta)$ domain with boundary $\partial\Omega$ a $d$-set. For the proof, we refer to \cite[Theorem 1, Chapter VII]{JoWa}.

\begin{prop}\label{teotraccia} Let $\frac{N-d}{2}<s<1$ and $\alpha$ be as in \eqref{definizione alpha}. $B^{2,2}_\alpha(\partial\Omega)$ is the trace space of $H^{s}(\Omega)$ in the following sense:
\begin{enumerate}%\label{enu1}
\item[(i)] $\gamma_0$ is a continuous linear operator from $H^s(\Omega)$ to $B^{2,2}_\alpha(\partial\Omega)$;
\item[(ii)] there exists a continuous linear operator $\Ext$ from $B^{2,2}_\alpha(\partial\Omega)$ to $H^{s}(\Omega)$ such that $\gamma_0\circ \Ext$ is the identity operator in $B^{2,2}_\alpha(\partial\Omega)$.
\end{enumerate}
\end{prop}

\noindent We point out that, if $\Omega\subset\R^N$ is a Lipschitz domain, its boundary $\partial\Omega$ is a $(N-1)$-set. Hence, the trace space of $H^s(\Omega)$ is $B^{2,2}_{s-\frac{1}{2}}(\partial\Omega)$, and the latter space coincides with $H^{s-\frac{1}{2}}(\partial\Omega)$.

The following result provides us with an equivalent norm on $H^s(\Omega)$. The proof can be achieved by adapting the proof of \cite[Theorem 2.3]{warmaCPAA}.
\begin{theorem}\label{equivalenza norme} Let $\Omega\subset\R^N$ be a $(\varepsilon,\delta)$ domain having as boundary a $d$-set, and let $\frac{N-d}{2}<s<1$. Then there exists a positive constant $C=C(\Omega,N,s,d)$ such that for every $u\in H^s(\Omega)$
\begin{equation}\label{norma eq}
\int_\Omega |u|^2\,\de\La_N\leq C\left(\frac{C_{N,s}}{2}\iint_{\Omega\times\Omega} \frac{|u(x)-u(y)|^2}{|x-y|^{N+2s}}\,\de\La_N(x)\de\La_N(y)+\int_{\partial\Omega}|u|^2\,\de\mu\right).
\end{equation}
\end{theorem}

\noindent Here, $C_{N,s}$ is the positive constant defined in Section \ref{sezgreen}. Hence, from Theorem \ref{equivalenza norme} and Proposition \ref{teotraccia}, the following norm is equivalent to the \lq\lq usual" $H^s(\Omega)$-norm:
\begin{equation}\notag
|||u|||^2_{H^s(\Omega)}:=\frac{C_{N,s}}{2}\iint_{\Omega\times\Omega} \frac{|u(x)-u(y)|^2}{|x-y|^{N+2s}}\,\de\La_N(x)\de\La_N(y)+\int_{\partial\Omega}|u|^2\,\de\mu.
\end{equation}

\noindent Finally, we recall the following important extension property which holds for $(\varepsilon,\delta)$ domains having as boundary a $d$-set. For details, we refer to Theorem 1, page 103 and Theorem 3, page 155 in \cite {JoWa}.
\begin{theorem}\label{teo estensione} Let $0<s<1$. There exists a linear extension operator $\mathcal{E}{\rm xt}\colon\,H^s(\Omega)\to\,H^{s}(\R^N)$ such that
\begin{equation}\label{R-3d}
\|\mathcal{E}{\rm xt}\, w\|^2_{H^{s}(\R^N)}\leq C\|w\|^2_{H^{s}(\Omega)},
\end{equation}
where $C$ is a positive constant depending on $s$, where $\mathcal{E}{\rm xt}\, w=w$ on $\Omega$.
\end{theorem}

\medskip

\noindent Domains $\Omega$ satisfying property \eqref{R-3d} are the so-called \emph{$H^s$-extension domains}.

\subsection{Sobolev embeddings and the Nash inequality}\label{sobolev embed}

We now recall some important Sobolev-type embeddings for the fractional Sobolev space $H^s(\Omega)$ where $\Omega$ is a $H^{s}$-extension domain with boundary a $d$-set, see \cite[Theorem 6.7]{hitch} and \cite[Lemma 1, p. 214]{JoWa} respectively. 

\noindent We set $$2^*:=\frac{2N}{N-2s}\quad\text{and}\quad\bar{2}:=\frac{2d}{N-2s}.$$
 
\begin{theorem}\label{immsobpstar} Let $s\in(0,1)$ such that $2s<N$. Let $\Omega\subseteq\R^N$ be a $H^s$-extension domain. Then $H^s(\Omega)$ is continuously embedded in $L^q(\Omega)$ for every $q\in [1,2^*]$, i.e. there exists a positive constant $C_1=C_1(N,s,\Omega)$ such that, for every $u\in H^{s}(\Omega)$,
\begin{equation}\label{sobolev p star}
\|u\|_{L^q(\Omega)}\leq C_1\|u\|_{H^s(\Omega)}.
\end{equation}
\end{theorem}

\begin{theorem}\label{immsobpbar} Let $s\in(0,1)$ be such that $N-d<2s<N$. Let $\Omega\subseteq\R^N$ be a $H^s$-extension domain having as boundary $\partial\Omega$ a $d$-set, for $0<d\leq N$. Then $H^s(\Omega)$ is continuously embedded in $L^q(\partial\Omega)$ for every $q\in [1,\bar{2}\,]$, i.e. there exists a positive constant $C_2=C_2(N,s,d,\Omega)$ such that, for every $u\in H^s(\Omega)$,
\begin{equation}\label{sobolev p bar}
\|u\|_{L^q(\partial\Omega)}\leq C_2\|u\|_{H^s(\Omega)}.
\end{equation}
\end{theorem}

\noindent We point out that $2^*\geq\bar{2}\geq 2$.

For $1\leq p\leq\infty$, we denote by $L^p(\Omega,m)$ the Lebesgue space with respect to the measure
\begin{equation}\label{defmisura}
\de m=\de\La_N+\de\mu,
\end{equation}
where $\mu$ is the $d$-measure supported on $\partial\Omega$. For $p\in[1,\infty)$, we endow $L^p(\Omega,m)$ with the following norm:
\begin{equation*}
\|u\|^p_{L^p(\Omega,m)}=\|u\|^p_{L^p(\Omega)}+\|\uu\|^p_{L^p(\partial\Omega)}.
\end{equation*}
For $p=\infty$, we endow $L^\infty(\Omega,m)$ with the following norm
$$\|u\|_{L^\infty(\Omega,m)}:=\max\left\{\|u\|_{L^\infty(\Omega)},\|\uu\|_{L^\infty(\partial\Omega)}\right\}.$$
With these definitions, $L^p(\Omega,m)$ becomes a Banach space for every $1\leq p\leq\infty$.

We now prove a version of the well known Nash inequality adapted to our setting.
\begin{prop}\label{Nash}
Let $u\in H^s(\Omega)$. Then there exists a positive constant $\bar{C}=\bar{C}(N,s,d,\Omega)$ such that the following Nash inequality holds,
\begin{equation}\label{Nashineq}
\|u\|_{L^2(\Omega,m)}^{2+\frac{4}{\lambda}}\leq\bar{C}\|u\|^2_{H^s(\Omega)}\|u\|_{L^1(\Omega,m)}^{\frac{4}{\lambda}},
\end{equation}
where $\lambda=\frac{2d}{d-N+2s}$.
\end{prop}

\begin{proof} We adapt the proof of Proposition 4.5 in \cite{Daners} to our context. We set $\lambda=\frac{2d}{d-N+2s}$. From interpolation inequalities (see e.g. \cite[Section 7.1]{giltrud}), we have that
\begin{equation}\label{interpolazione}
\|u\|_{L^2(\omega)}\leq\|u\|_{L^{\bar{2}}(\omega)}^{1-\mu}\|u\|_{L^1(\omega)}^\mu,
\end{equation}
with $\mu=\frac{d-N+2s}{2d-N+2s}=1-\frac{d}{2d-N+2s}$ and $\omega$ is either $\Omega$ or $\partial\Omega$. Hence, from Theorems \ref{immsobpstar} and \ref{immsobpbar} we obtain
\begin{equation}\label{somma1}
\|u\|_{L^2(\Omega,m)}\leq C_3\|u\|_{H^s(\Omega)}^{\frac{d}{2d-N+2s}}\|u\|_{L^1(\Omega,m)}^{1-\frac{d}{2d-N+2s}},
\end{equation}
where $C_3=C_3(N,s,d,\Omega)=\max\{C_1,C_2\}$ and $C_1$ and $C_2$ are the constants appearing in \eqref{sobolev p star} and \eqref{sobolev p bar} respectively. Therefore, since $\frac{2d-N+2s}{d}=\frac{1}{1-\mu}=1+\frac{2}{\lambda}$, from Theorem \ref{equivalenza norme} we have that there exists a positive constant $\bar{C}$ depending on $N$, $s$, $d$ and $\Omega$ such that
\begin{equation}\notag%\label{Nashineq}
\|u\|_{L^2(\Omega,m)}^{1+\frac{2}{\lambda}}\leq\bar{C}\|u\|_{H^s(\Omega)}\|u\|_{L^1(\Omega,m)}^{\frac{2}{\lambda}},
\end{equation}
i.e., \eqref{Nashineq} holds.
\end{proof}

\section{The time dependent generalized regional fractional Laplacian}\label{sezgreen}
\setcounter{equation}{0}

From now on, let $T>0$ be fixed. We introduce a suitable measurable function $K\colon [0,T]\times\Omega\times\Omega\to\R$ such that $K(t,\cdot,\cdot)$ is symmetric for every $t\in[0,T]$ and there exist two constants $0<k_1<k_2$ such that $k_1\leq K(t,x,y)\leq k_2$ for a.e. $t\in[0,T]$ and $x,y\in\Omega$.\\
For $u,v\in H^s(\Omega)$, we set
\begin{equation}\notag
|u|_{s,K}:=\frac{C_{N,s}}{2}\iint_{\Omega\times\Omega}K(t,x,y)\frac{(u(x)-u(y))^2}{|x-y|^{N+2s}}\,\de\La_N(x)\de\La_N(y)
\end{equation}
and
\begin{equation}\notag
(u,v)_{s,K}:=\frac{C_{N,s}}{2}\iint_{\Omega\times\Omega}K(t,x,y)\frac{(u(x)-u(y))(v(x)-v(y))}{|x-y|^{N+2s}}\,\de\La_N(x)\de\La_N(y).
\end{equation}
We point out that $|u|_{s,1}=|u|^2_{H^s(\Omega)}$ and $(u,v)_{s,1}=(u,v)_s$.

% There is a huge literature on the fractional Laplace operator; among the others, we refer to \cite{dubspa,CSF,mandelbrot,schneider,valdinoci}.
We introduce now a time dependent generalization of the regional fractional Laplacian $(-\Delta)^s_\Omega$. For the definition of regional fractional Laplacian, among the others we refer to \cite{BBC,chenkum,guan1,guan2,guan3}.

\noindent Let $s\in(0,1)$. For every fixed $t\in[0,T]$, we define the operator $\B$ acting on $H^s(\Omega)$ in the following way:
\begin{equation}\label{fracreglap}
\begin{split}
\B u(t,x) &=C_{N,s}{\rm P.V.}\int_\Omega K(t,x,y) \frac{u(t,x)-u(t,y)}{|x-y|^{N+2s}}\,\de\La_N(y)\\[2mm]
&=C_{N,s}\lim_{\epsilon\to 0^+}\int_{\{y\in\Omega\,:\,|x-y|>\epsilon\}} K(t,x,y)\frac{u(t,x)-u(t,y)}{|x-y|^{N+2s}}\,\de\La_N(y).
\end{split}
\end{equation}
The positive constant $C_{N,s}$ is defined by
\begin{equation}\notag
C_{N,s}=\frac{s2^{2s}\Gamma(\frac{N+2s}{2})}{\pi^\frac{N}{2}\Gamma(1-s)},
\end{equation}
where $\Gamma$ is the Euler function.
If $K\equiv 1$ on $[0,T]\times\Omega\times\Omega$, the operator $\B$ reduces to the usual regional fractional Laplacian $(-\Delta)^s_\Omega$.\\

%We now give a suitable definition of $p$-fractional normal derivative on non-smooth domains. In the literature (see \cite{warmaNODEA,warmaDtoN}) there are different definitions of fractional normal derivatives on Lipschitz domains. Our aim is to prove a "$p$-fractional" Green formula, following the philosophy of Section 3 in \cite{JEEfraz}, for domains with fractal boundary. A key tool is the use of a limit argument since our fractal domain $Q$ is approximated by the increasing sequence of pre-fractal domains $Q_n$.\\

We now introduce the notion of \emph{fractional conormal derivative} on $(\varepsilon,\delta)$ domains having as boundary a $d$-set and satisfying hypotheses $(\Hcal)$ in Section \ref{geometria}. We will generalize the notion of fractional normal derivative on irregular sets, which was introduced in \cite{JEEfraz} (see also \cite{CLVmosco,CLNODEA} for the nonlinear case).

\noindent We define the space
\begin{equation*}
V(\B,\Omega):=\{u\in H^s(\Omega)\,:\,\B u\in L^{2}(\Omega)\,\,\text{in the sense of distributions}\},
\end{equation*}
which is a Banach space equipped with the norm
\begin{equation}\notag
\|u\|^2_{V(\B,\Omega)}:=\|u\|^2_{H^s(\Omega)}+\|\B u\|^2_{L^{2}(\Omega)}.
\end{equation}

\noindent We define the fractional conormal derivative on Lipschitz domains.

%thus, the definition of fractional normal derivative given in Definition 2.9 (b) of \cite{galwarmafrac} does not apply in our case.
\begin{definition}\label{derivatafrazlip}
Let $\T\subset\R^N$ be a Lipschitz domain. Let $u\in V(\BT,\T):=\{u\in H^s(\T)\,:\,\BT u\in L^{2}(\T)\,\,\text{in the sense of distributions}\}$. We say that $u$ has a weak fractional conormal derivative in $(H^{s-\frac{1}{2}}(\partial\T))'$ if there exists $g\in (H^{s-\frac{1}{2}}(\partial\T))'$ such that
\begin{align}
&\left\langle g,\vv\right\rangle_{(H^{s-\frac{1}{2}}(\partial\T))', H^{s-\frac{1}{2}}(\T)}
=-\int_{\T} \BT u\,v\,\de\La_N \label{fracgreenlip}\\[2mm]
&+\frac{C_{N,s}}{2}\iint_{\T\times\T}K(t,x,y)\frac{(u(x)-u(y))(v(x)-v(y))}{|x-y|^{N+2s}}\,\de\La_N(x)\de\La_N(y)\notag
\end{align}
for every $v\in H^s(\T)$. In this case, $g$ is uniquely determined and we call $C_{s}\Ncal u:=g$ the weak fractional conormal derivative of $u$, where
\begin{equation*}
C_s:=\frac{C_{1,s}}{2s(2s-1)}\int_0^\infty\frac{|z-1|^{1-2s}-(z\vee 1)^{1-2s}}{z^{2-2s}}\,\de z.
\end{equation*}
\end{definition}

\medskip

\noindent We point out that, if $K(t,x,y)\equiv1$, we recover the definition of fractional normal derivative on Lipschitz sets introduced in \cite{JEEfraz}. Moreover, if in addition to that we let $s\to 1^-$ in \eqref{fracgreenlip}, we obtain the Green formula for Lipschitz domains \cite{bregil}.

\medskip

\begin{theorem}[Generalized fractional Green formula]\label{greenf}
There exists a bounded linear operator $\Ncal$ from $V(\B,\Omega)$ to $(B^{2,2}_{\alpha}(\partial\Omega))'$.\\
The following generalized Green formula holds for every $u\in V(\B,\Omega)$ and $v\in H^s(\Omega)$,
\begin{equation}\label{fracgreen}
\begin{split}
&C_{s}\left\langle \Ncal u,\vv\right\rangle_{(B^{2,2}_{\alpha}(\partial\Omega))', B^{2,2}_{\alpha}(\partial\Omega)}
=-\int_\Omega \B u\,v\,\de\La_N\\[4mm]
&+\frac{C_{N,s}}{2}\iint_{\Omega\times\Omega}K(t,x,y)\frac{(u(x)-u(y))(v(x)-v(y))}{|x-y|^{N+2s}}\,\de\La_N(x)\de\La_N(y).
\end{split}
\end{equation}

%Moreover, the operator $u \mapsto l(u)=\Ncal u$ is linear and continuous on $\tilde B^{2,2}_s(\partial Q)$.
\end{theorem}

\begin{proof} We adapt to our setting the proof of Theorem 2.2 in \cite{CLNODEA}, which we recall for the sake of completeness.
For $u\in V(\B,\Omega)$ and $v\in H^s(\Omega)$, we define
\begin{equation}\notag
%\begin{split}
\langle l(u),v\rangle:=-\int_\Omega \B u\,v\,\de\La_N +\frac{C_{N,s}}{2}(u,v)_{s,K}.
%\end{split}
\end{equation}

\noindent From H\"older's inequality, the trace theorem and the hypotheses on $K(t,x,y)$, we get
\begin{align}
|\left\langle l(u),v\rangle\right|&\leq \|\B u\|_{L^{2}(\Omega)}\|v\|_{L^2(\Omega)}+k_2\frac{C_{N,s}}{2}\|u\|_{H^s(\Omega)}\|v\|_{H^{s}(\Omega)}\notag\\
&\leq C\,\|u\|_{V(\B,\Omega)}\|v\|_{H^s(\Omega)}\leq C\,\|u\|_{V(\B,\Omega)}\|v\|_{B^{2,2}_{\alpha}(\partial\Omega)}.\label{eqstima}
\end{align}
We prove that the operator $l(u)$ is independent from the choice of $v$ and it is an element of $(B^{2,2}_{\alpha}(\partial\Omega))'$. From Proposition \ref{teotraccia}, for every $v\in B^{2,2}_{\alpha}(\partial\Omega)$ there exists a function $\tilde w:=\Ext\,v\in H^s(\Omega)$ such that 
\begin{equation}\label{eq1}
\|\tilde w\|_{H^{s}(\Omega)}\leq C\|v\|_{B^{2,2}_{\alpha}(\partial\Omega)}
\end{equation}
and $\tilde w|_{\partial\Omega}=v$ $\mu$-almost everywhere. From \eqref{fracgreen} we have that
\begin{equation}\notag
C_{s}\left\langle \Ncal u,v\right\rangle_{(B^{2,2}_{\alpha}(\partial\Omega))', B^{2,2}_{\alpha}(\partial\Omega)}=\langle l(u),\tilde w\rangle.
\end{equation}
The conclusion follows from \eqref{eqstima} and \eqref{eq1}.

We now recall that $\Omega$ is approximated by a sequence of Lipschitz domains $\Omega_n$, for $n\in\N$, satisfying conditions $(\Hcal)$ in Section \ref{geometria}. From \eqref{fracgreenlip} we have that
\begin{align}
&C_{s}\left\langle \Ncal u,\vv\right\rangle_{(H^{s-\frac{1}{2}}(\partial\Omega_n))', H^{s-\frac{1}{2}}(\partial\Omega_n)}
=-\int_\Omega\chi_{\Omega_n}\Bn u\,v\,\de\La_N \notag\\[2mm]%\label{fracgreenlip}
&+\frac{C_{N,s}}{2}\iint_{\Omega\times\Omega}\chi_{\Omega_n}(x)\chi_{\Omega_n}(y)K(t,x,y) \frac{(u(x)-u(y))(v(x)-v(y))}{|x-y|^{N+2s}}\,\de\La_N(x)\de\La_N(y).\notag
\end{align}
From the dominated convergence theorem, we have
\begin{align*}\label{E:Lipschitz}
&\lim_{n\to\infty} C_{s}\left\langle \Ncal u,\vv\right\rangle_{(H^{s-\frac{1}{2}}(\partial\Omega_n))', H^{s-\frac{1}{2}}(\partial\Omega_n)}=\lim_{n\to\infty}\left(-\int_{\Omega_n} \Bn u\,v\,\de\La_N\right.\\[2mm]\notag
&+\left.\frac{C_{N,s}}{2}\iint_{\Omega_n\times\Omega_n}K(t,x,y)\frac{(u(x)-u(y))(v(x)-v(y))}{|x-y|^{N+2s}}\,\de\La_N(x)\de\La_N(y)\right)\notag\\
&=-\int_\Omega \B u\,v\,\de\La_N +\frac{C_{N,s}}{2}(u,v)_{s,K}=\left\langle l(u),v\right\rangle\notag
\end{align*}
for every $u\in V(\B,\Omega)$ and $v\in H^{s}(\Omega)$. Hence, we define the fractional conormal derivative on $\Omega$ as
\begin{equation}\notag
\langle C_{s}\Ncal u,\vv\rangle_{(B^{2,2}_{\alpha}(\partial\Omega))', B^{2,2}_{\alpha}(\partial\Omega)}:=-\int_\Omega \B u\,v\,\de\La_N+\frac{C_{N,s}}{2}(u,v)_{s,K}.
\end{equation}
\end{proof}

\begin{remark} As in the Lipschitz case, when $K(t,x,y)\equiv 1$, we recover the notion of fractional normal derivative on an irregular set introduced in \cite{JEEfraz}.
Moreover, from \cite[Remark 3.1]{JEEfraz}, when $s\to 1^-$ and $K(t,x,y)\equiv1$ in \eqref{fracgreen}, we recover the Green formula proved in \cite{LaVe2} for fractal domains.%\\
%Let $u\in V(-\Delta_p, \Omega):=\{u\in W^{1,p}(\Omega)\,:\,-\Delta_p u\in L^{p'}(\Omega)\,\,\text{in the sense of distributions}\}$ and $v\in W^{1,p}(\Omega)$. It holds that
%\begin{equation*}
%\lim_{s\to 1^-} \int_\Omega (-\Delta_p)_\Omega^s u\,v\,\de\La_N=\int_\Omega|\nabla u|^{p-2}\nabla u\,\nabla v\,\de\La_N.
%\end{equation*}
%As first step, we take $v=u$ and $u\in C^\infty(\overline\Omega)$. In particular then $u\in C^\infty(\overline\Omega_n)$ for every $n$ and $\Ncal u=0$ on $\partial\Omega_n$ pointwise (see \cite{warmaNODEA}). From Definition \ref{derivatafrazlip} we have
%\begin{equation}\label{green limite}
%\lim_{s\to 1^-} \int_\Omega \chi_{\Omega_n}u\,(-\Delta_p)_\Omega^s u\,\de\La_N=\lim_{s\to 1^-}\frac{(1-s)C_{N,p,s}}{2(1-s)}\iint_{\Omega_n\times\Omega_n}\frac{|u(x)-u(y)|^p}{|x-y|^{N+sp}}\,\de\La_N(x)\de\La_N(y).
%\end{equation}
%From \cite{BBM1,BBM2} and the properties of the Euler function, the limit in the right-hand side of \eqref{green limite} is equal to $\displaystyle\int_{\Omega_n}|\nabla u|^p\,\de\La_N$. Then passing to the limit as $n\to+\infty$ we get
%\begin{equation*}
%\lim_{n\to+\infty}\lim_{s\to 1^-} \int_\Omega \chi_{\Omega_n}u\,(-\Delta_p)_\Omega^s u\,\de\La_N=\lim_{n\to+\infty}\int_{\Omega_n}|\nabla u|^p\,\de\La_N=\int_\Omega |\nabla u|^p\,\de\La_N.
%\end{equation*}
%The claim then follows by density arguments.
\end{remark}

%We remark that, when $s\to 1^-$ in \eqref{fracgreen}, we recover the Green formula stated in \cite{LaVe2} for fractal domains.

\section{The non-autonomous energy form and the evolution family}\label{sec3}
\setcounter{equation}{0}

%\subsection{The energy form}\label{funzionali}

From now on, let us suppose that $s\in (0,1)$ is such that $N-d<2s<N$. Let $b\colon(0,T)\times\partial\Omega\to\R$ be a function such that
\begin{equation}\label{ipotesi b}
\begin{cases}
b\in L^\infty([0,T]\times\partial\Omega),\\
\inf b(t,P)>b_0>0\quad \text{for every }(t,P)\in [0,T]\times\partial\Omega,\\
\text{there exists }\eta\in(\frac{1}{2}, 1): |b(t,P)-b(\tau,P)|\leq c |t-\tau|^\eta\quad\text{for every }P\in\partial\Omega.
\end{cases}
\end{equation}

\noindent Let $\zeta\colon[0,T]\times\partial\Omega\times\partial\Omega\to\R$ be such that $\zeta(t,\cdot,\cdot)$ is symmetric for every fixed $t\in[0,T]$ and $\zeta_1\leq\zeta(t,x,y)\leq\zeta_2$ for suitable constants $0<\zeta_1<\zeta_2$ and for a.e $(t,x,y)\in[0,T]\times\partial\Omega\times\partial\Omega$.

We now introduce a bounded linear operator $\Theta^t_{\alpha}\colon B^{2,2}_\alpha (\partial\Omega)\to (B^{2,2}_\alpha (\partial\Omega))'$ defined by
\begin{equation}\label{nonlocal-op.}
\langle\Theta^t_{\alpha}(u),v\rangle:=\iint_{\partial\Omega\times\partial\Omega}\zeta(t,x,y)\frac{(u(x)-u(y))(v(x)-v(y))}{|x-y|^{d+2\alpha}}\,\de\mu(x)\,\de\mu(y),
\end{equation}
where $\langle\cdot,\cdot\rangle$ denotes the duality pairing between $(B^{2,2}_\alpha (\partial\Omega))'$ and $B^{2,2}_\alpha (\partial\Omega)$ and $\alpha$ is defined in \eqref{definizione alpha}. From our hypotheses on $\zeta$, this nonlocal term on $\partial\Omega$ is equivalent to the seminorm of $B^{2,2}_\alpha (\partial\Omega)$; moreover, we point out that, if $\zeta\equiv 1$, it can be regarded as a regional fractional Laplacian of order $\alpha$ on the boundary.

\medskip

%VA PENSATO SE INTRODURLO COSI' O COME L'OPERATORE $\B$!!!

\noindent We suppose that the kernels $K(t,x,y)$ and $\zeta(t,x,y)$ appearing in the nonlocal terms in the bulk and on the boundary respectively are H\"older continuous with respect to $t$. More precisely, we suppose that there exists $\eta\in(\frac{1}{2},1)$ such that
\begin{equation}\label{holderianita nuclei}
|K(t,x,y)-K(\tau,x,y)|\leq C|t-\tau|^\eta\quad\text{and}\quad|\zeta(t,x,y)-\zeta(\tau,x,y)|\leq C|t-\tau|^\eta
\end{equation}
for a.e. $x,y\in\Omega$ and a.e. $x,y\in\partial\Omega$ respectively.
Obviously, one can take different H\"older exponents in \eqref{ipotesi b} and \eqref{holderianita nuclei}. For the sake of simplicity, we suppose that the third condition in \eqref{ipotesi b} and \eqref{holderianita nuclei} hold for the same exponent $\eta\in(\frac{1}{2},1)$.

\medskip

\noindent For every $u\in H:=L^2(\Omega,m)$, we introduce the following energy form, with effective domain $D(E)=[0,T]\times H^s(\Omega)$,
\begin{equation}\label{frattale}
E[t,u]:=
\begin{cases}
\,\displaystyle\frac{C_{N,s}}{2}\iint_{\Omega\times\Omega}K(t,x,y)\frac{|u(x)-u(y)|^2}{|x-y|^{N+2s}}\,\de\La_N(x)\de\La_N(y)\\[2mm]
+\displaystyle\int_{\partial\Omega} b(t,P)|\uu|^2\,\de\mu+\langle\Theta^t_{\alpha}(\uu),\uu\rangle&\text{if}\,\,u\in D(E),\\[2mm]
+\infty &\text{if}\,\,u\in H\setminus D(E),
\end{cases}
\end{equation}

\noindent We remark that, if $u\in H^s(\Omega)$, from \eqref{definizione alpha} and Proposition \ref{teotraccia}, its trace $\uu$ is well-defined.

We now prove some properties of the form $E$.

\begin{prop}\label{coer} For every $t\in[0,T]$, the form $E[t,u]$ is continuous and coercive on $H^s(\Omega)$.
\end{prop}

\begin{proof} We start by proving the continuity of $E$. Since $b\in L^\infty([0,T]\times\partial\Omega)$ and $K$ and $\zeta$ are bounded from above, for every $t\in[0,T]$ we have
\begin{equation}\notag
\begin{split}
&E[t,u]\leq k_2|u|^2_{H^s(\Omega)}+\|b\|_{L^\infty([0,T]\times\partial\Omega)}\|u\|^2_{L^2(\partial\Omega)}+\zeta_2\iint_{\partial\Omega\times\partial\Omega}\frac{(u(x)-u(y))^2}{|x-y|^{d+2\alpha}}\,\de\mu(x)\,\de\mu(y)\\[2mm]
&\leq k_2|u|^2_{H^s(\Omega)}+\max\{\|b\|_{L^\infty([0,T]\times\partial\Omega)},\zeta_2\}\|u\|^2_{B^{2,2}_\alpha(\partial\Omega)}\\[2mm]
&\leq\max\{k_2,C\|b\|_{L^\infty([0,T]\times\partial\Omega)},C\zeta_2\}\,\|u\|^2_{H^s(\Omega)},
\end{split}
\end{equation}
where the last inequality follows from the trace theorem.\\
We now prove the coercivity. By using again the hypotheses on $b$, $K$ and $\zeta$, for every $t\in[0,T]$ we have
\begin{equation}\notag
\begin{split}
&E[t,u]\geq k_1|u|^2_{H^s(\Omega)}+b_0\|u\|^2_{L^2(\partial\Omega)}+\langle\Theta^t_\alpha (u),u\rangle\\[2mm]
&\geq\min\{k_1,b_0\}\left(|u|^2_{H^s(\Omega)}+\|u\|^2_{L^2(\partial\Omega)}\right)\geq\beta\|u\|^2_{H^s(\Omega)},
\end{split}
\end{equation}
for a suitable constant $\beta>0$, where the last inequality follows from Theorem \ref{equivalenza norme}.
\end{proof}

\begin{prop}\label{chiusura} For every $t\in[0,T]$, $E[t,u]$ is closed on $L^2(\Omega,m)$.
\end{prop}

\begin{proof}
For every fixed $t\in[0,T]$, we have to prove that for every sequence $\{u_k\}\subseteq H^s(\Omega)$ such that
\begin{equation}\label{cauchyseq}
E[t,u_k-u_j]+\|u_k-u_j\|_{L^2(\Omega,m)}\to 0\quad\text{for}\,\,k,j\to +\infty,
\end{equation}
there exists $u\in H^s(\Omega)$ such that
\begin{center}
$\displaystyle E[t,u_k-u]+\|u_k-u\|_{L^2(\Omega,m)}\to 0\quad$ for $k\to +\infty$.
\end{center}
This means that we should prove that
\begin{equation}\label{tesi chiusura}
\begin{split}
&|u_k-u|_{s,K}+\int_{\partial\Omega}b(t,P)|u_k-u|^2\,\de\mu\\[2mm]
&+\langle\Theta^t_{\alpha}(u_k-u),u_k-u\rangle+\|u_k-u\|_{L^2(\Omega,m)}\to 0\quad\text{for }k\to +\infty.
\end{split}
\end{equation}
We point out that \eqref{cauchyseq} infers that $\{u_k\}$ is a Cauchy sequence in $L^2(\Omega,m)$ and, since $L^2(\Omega,m)$ is a Banach space, there exists $u\in L^2(\Omega,m)$ such that
\begin{center}
$\|u_k-u\|_{L^2(\Omega,m)}\xrightarrow[k\to +\infty]{} 0$.
\end{center}
Moreover, since $|u_k-u_j|_{s,K}+\|u_k-u_j\|_{L^2(\Omega)}$ is equivalent to the $H^s(\Omega)$-norm of $u_k-u_j$, \eqref{cauchyseq} implies that $\{u_k\}$ is a Cauchy sequence also in $H^s(\Omega)$. Since $H^s(\Omega)$ is a Banach space, then also $\|u_k-u\|^2_{H^s(\Omega)}\to 0$ when $k\to+\infty$. Hence, since $|u_k-u|_{s,K}\leq k_2|u_k-u|^2_{H^s(\Omega)}$, the first term on the left-hand side of \eqref{tesi chiusura} vanishes as $k\to+\infty$.

\noindent From hypotheses \eqref{ipotesi b}, for every fixed $t\in [0,T]$ the thesis follows from \cite[Proposition 4.1]{JEEfraz} for the second term on the left-hand side of \eqref{tesi chiusura}. As to the term $\langle\Theta^t_{\alpha}(u_k-u),u_k-u\rangle$, we point out that from the hypotheses on $\zeta$ and the trace theorem we have 
\begin{equation}\notag
\begin{split}
\langle\Theta^t_\alpha (u_k-u),u_k-u\rangle&=\iint_{\partial\Omega\times\partial\Omega}\zeta(t,x,y)\frac{|u_k(x)-u(x)-(u_k(y)-u(y))|^2}{|x-y|^{d+2\alpha}}\,\de\mu(x)\,\de\mu(y)\\[2mm]
&\leq\zeta_2\|u_k-u\|^2_{B^{2,2}_\alpha (\partial\Omega)}\leq\,C\|u_k-u\|^2_{H^s(\Omega)},
\end{split}
\end{equation}
and the last term tends to 0 when $k\to +\infty$ because $u_k\to u$ in $H^s(\Omega)$.
\end{proof}

\begin{theorem}\label{dirform} For every $t\in[0,T]$, $E[t,u]$ is Markovian, hence it is a Dirichlet form on $L^2(\Omega,m)$.
\end{theorem}

\medskip

\noindent The proof follows by adapting the one of Theorem 3.4 in \cite{nostroMMAS}, see also \cite[Lemma 2.7]{warma2015}.

\bigskip

%\subsection{The evolution families}\label{sezionefamevolutive}

\noindent By $E(t,u,v)$ we denote the corresponding bilinear form
\begin{equation}\label{frattalebilineare}
E(t,u,v)=(u,v)_{s,K}+\int_{\partial\Omega} b(t,P) \uu\,\vv \, \de\mu+\langle\Theta^t_{\alpha}(\uu),\vv\rangle
\end{equation}
defined on $[0,T]\times H^s(\Omega)\times H^s(\Omega)$.

\begin{theorem} For every $u,v\in H^s(\Omega)$ and for every $t\in [0,T]$, $E(t,u,v)$ is a closed symmetric bilinear form on $L^2(\Omega,m)$. Then there exists a unique selfadjoint non-positive operator $A(t)$ on $L^2(\Omega,m)$ such that
\begin{equation}\label{corr}
E(t,u,v)=(-A(t)u,v)_{L^2(\Omega,m)}\quad\text{for every }u\in D(A(t)),\,v\in H^s(\Omega),
\end{equation}
where $D(A(t))\subset H^s(\Omega)$ is the domain of $A(t)$ and it is dense in $L^2(\Omega,m)$.
\end{theorem}

\medskip

\noindent For the proof we refer to Theorem 2.1, Chapter 6 in \cite{kato}.

\begin{prop}\label{propertiesenergyform}
For every $t\in[0,T]$, the form $E(t,u,v)$ has the square root property, i.e. $D(A(t))^{\frac{1}{2}}=H^s(\Omega)$. Moreover, there exists a constant $C>0$ such that, for every $\eta\in (\frac{1}{2},1)$,
\begin{equation}\label{holderianita forma in t}
 |E(t,u,v)-E(\tau,u,v)|\leq C|t-\tau|^\eta\|u\|_{H^s(\Omega)}\|v\|_{H^s(\Omega)}, \quad 0\leq\tau,t\leq T.
\end{equation}
\end{prop}

\begin{proof} The square root property follows since the form is symmetric and bounded. As to \eqref{holderianita forma in t}, we have that
\begin{equation}\notag
\begin{split}
&|E(t,u,v)-E(\tau,u,v)|\\[2mm]
&\leq\frac{C_{N,s}}{2}\iint_{\Omega\times\Omega}|K(t,x,y)-K(\tau,x,y)|\frac{|u(x)-u(y)||v(x)-v(y)|}{|x-y|^{N+2s}}\,\de\La_N(x)\de\La_N(y)\\[2mm]
&+\int_{\partial\Omega} |b(t,P)-b(\tau,P)|\,|u(P)|\,|v(P)|\,\de\mu\\[2mm]
&+\iint_{\partial\Omega\times\partial\Omega}|\zeta(t,x,y)-\zeta(\tau,x,y)|\frac{|u(x)-u(y)||v(x)-v(y)|}{|x-y|^{d+2\alpha}}\,\de\mu(x)\de\mu(y)\\[2mm]
&\leq C|t-\tau|^\eta\|u\|_{H^s(\Omega)}\|v\|_{H^s(\Omega)},
\end{split}
\end{equation}
where the last inequality follows from the hypotheses \eqref{ipotesi b} on $b$, from \eqref{holderianita nuclei} and the trace theorem.
\end{proof}

\begin{prop}\label{propertiesoperatorA}
For every $t \in [0,T]$ and $\tau\geq 0$, $A(t)\colon\D(A(t))\to L^2(\Omega,m)$ is the generator of a semigroup $e^{\tau A(t)}$ on $L^2(\Omega,m)$ which is strongly continuous, contractive and analytic with angle $\omega_{A(t)}>0$.
\end{prop}

\begin{proof}
The analyticity follows from the coercivity of $E[t,u]$ (see Theorem 6.2, Chapter 4 in \cite{showalter}). The contraction property follows from Lumer-Phillips Theorem (see Theorem 4.3, Chapter 1 in \cite{pazy}). The strong continuity follows from Theorem 1.3.1 in \cite{fukush}
\end{proof}

%%\subsection{Structural assumptions}

\begin{proposition}\label{propertiesAstructutural}
For every $t\in [0,T]$, the operator $A(t)$ satisfies the following \mbox{properties}:
\begin{enumerate}
  \item[1)] the spectrum of $A(t)$ is contained in a sectorial open domain
	$$\sigma(A(t))\subset \Sigma_\omega=\{\mu\in \C\,:\,|\rm{Arg}\,\mu|<\omega\}$$ for some fixed angle $0<\omega<\frac{\pi}{2}$.
	The resolvent satisfies the estimate
	$$\|\left(\mu-A(t)\right)^{-1}\|_{\Lm(L^2(\Omega,m))}\leq\frac{M}{|\mu|}$$ for $M\geq 1$ independent from t and $\mu\notin \Sigma_\omega \cup 0$;
	moreover, $A(t)$ is invertible and $\|A(t)^{-1}\|\leq M_1$ with $M_1$ independent from $t$;

  \item[2)] $D(A(t))\subset D(A(\tau))^{\frac{1}{2}}=H^s(\Omega),\quad 0\leq\tau\leq t\leq T$; in particular, $D(A(t))\subset D(A(\tau))^{\nu}$ for every $\nu$ such that $0<\nu\leq \frac{1}{2}$;

  \item[3)] $A(t)^{-1}$ is H\"older continuous in $t$ in the sense of Yagi, i.e.,
  \begin{equation}\label{HolderA(t)-1}
  \left\|A(t)^{\frac{1}{2}} \left(A(t)^{-1}-A(\tau)^{-1}\right)\right\|_{\Lm(L^2(\Omega,m))}\leq C|t-\tau|^\eta,
  \end{equation}
  with some fixed exponent $\eta\in \left(\frac{1}{2},1\right]$ and $C>0$.
\end{enumerate}
\end{proposition}

%????NOI l'abbiamo provata con $\nu =1/2$ma serve con \nu????

\begin{proof} The first two properties follow from Propositions \ref{propertiesenergyform} and \ref{propertiesoperatorA}. In order to prove the H\"older continuity, one can proceed as in \cite[Chapter 3, section 7.1]{Yagi}
% SEGUO YAGI PAG 151 p=2 e pag 150:

Let $\mathcal{A}(t)\colon H^s(\Omega)\to H^{-s}(\Omega)$ denote the sectorial operator with angle $\omega_{\mathcal{A}(t)}\leq \omega_A <\frac{\pi}{2}$ associated with $E(t,u,v)$,
$$E(t,u,v)=-\langle\mathcal{A}(t)u, v\rangle_{H^{-s}(\Omega),H^s(\Omega)},$$
with $u\in D(\mathcal{A}(t))=H^s(\Omega)$.

Let $\phi\in H^{-s}(\Omega)$ and $u\in H^s(\Omega)$. We have that
\begin{equation}\notag
\begin{split}
&\langle\mathcal{A}(t)[\mathcal{A}(t)^{-1}-\mathcal{A}(\tau)^{-1}]\phi,u\rangle_{H^{-s}(\Omega),H^s(\Omega)}\\
=&-\langle[\mathcal{A}(t)-\mathcal{A}(\tau)]\mathcal{A}(\tau)^{-1}\phi,u\rangle_{H^{-s}(\Omega),H^s(\Omega)}\\
=&E(t,\mathcal{A}(\tau)^{-1}\phi, u)-E(\tau,\mathcal{A}(\tau)^{-1}\phi, u).
\end{split}
\end{equation}
From \eqref{holderianita forma in t}, we obtain
\begin{equation}\notag
\|\mathcal{A}(t)[\mathcal{A}(t)^{-1}-\mathcal{A}(\tau)^{-1}]\phi\|_{H^{-s}(\Omega)}\leq C|t-\tau|^\eta\|\mathcal{A}(\tau)^{-1}\|_{\Lm(H^{-s}(\Omega)\to H^{s}(\Omega))}\|\phi\|_{H^{-s}(\Omega)}.
\end{equation}
From \cite[page 190]{Laasri}, we have that $\|\mathcal{A}(\tau)^{-1}\|_{\Lm(H^{-s}(\Omega)\to H^{s}(\Omega))}\leq C$ for a suitable positive constant $C$. Hence, we conclude that $$\|\mathcal{A}(t)[\mathcal{A}(t)^{-1}-\mathcal{A}(\tau)^{-1}]\phi\|_{H^{-s}(\Omega)}\leq C|t-\tau|^\eta\|\phi\|_{H^{-s}(\Omega)}.$$
In order to prove condition \eqref{HolderA(t)-1}, we note that for every $z\in H^s(\Omega)$ it holds that $A^{\frac{1}{2}}(\cdot)z= \mathcal{A}^{-\frac{1}{2}}(\cdot)\mathcal{A}(\cdot)z$. Therefore, we have
\begin{equation}\notag
\begin{split}
(A(t)^{\frac{1}{2}}[A(t)^{-1}-A(\tau)^{-1}]\phi,u)_{L^2(\Omega,m)}= E\left(t,A(\tau)^{-1}\phi,({\mathcal A}(t)^{-\frac{1}{2}})'u\right)\\
-E\left(\tau,A(\tau)^{-1}\phi,({\mathcal A}(t)^{-\frac{1}{2}})'u\right).
\end{split}
\end{equation}
Since adjoint operators have the same norm, taking into account \cite[page 190]{Laasri}, condition \eqref{HolderA(t)-1} holds.
%\color[rgb]{red}{prova}
\end{proof}

%CONTROLLARE I DETTAGLI DI QUESTA DIMOSTRAZIONE!!!!

\medskip

\noindent From the above results we deduce the following.

\begin{theo}\label{theorem1} For every $t\in [0,T]$, let $A(t)\colon D(A(t))\to L^2(\Omega,m)$ be the linear unbounded operator defined in \eqref{corr}. Then there exists a unique family of evolution operators $U(t,\tau)\in \Lm(L^2(\Omega,m))$ such that
\begin{enumerate}\label{propU}
  \item[1)] $U(\tau,\tau)=\rm{Id},\quad 0\leq\tau\leq T$;
  \item[2)] $U(t,\tau)U(\tau,\sigma)=U(t,\sigma),\quad 0\leq\sigma\leq\tau\leq  t\leq T$;
  \item[3)] for every $0\leq\tau\leq t\leq T$ one has
		\begin{equation}\label{estimatesU}
		\|U(t,\tau)\|_{\Lm(L^2(\Omega,m))}\leq 1,
		\end{equation}
	and $U(t,\tau)$ for $0\leq\tau\leq t<T$ is a strongly contractive family on $L^2(\Omega,m)$;
  \item[4)] the map $t\mapsto U(t,\tau)$ is differentiable in $(\tau,T]$ with values in $\Lm(L^2(\Omega,m))$ and $\frac{\partial U(t,\tau)}{\partial t}= A(t)U(t,\tau)$;
	\item[5)] $A(t)U(t,\tau)$ is a $\Lm(L^2(\Omega,m))$-valued continuous function for $0\leq\tau<t\leq T$. Moreover, there exists a constant $C>0$ such that
	\begin{equation}\label{estimate AU}
	\|A(t)U(t,\tau)\|_{\Lm(L^2(\Omega,m))}\leq\frac{C}{t-\tau},\quad 0\leq\tau<t<T.
	\end{equation}
\end{enumerate}
\end{theo}

\noindent For the proof, see Section 5.3 in \cite{Yagi}.
%\begin{proof} We refer to section 5.3 in  \cite{Yagi} .

%TOGLIERE LA PROVA
%  As to the contractive property:
%For every $x\in L^2(\Omega, m)$

%$$\frac{d\|U(t,s)x\|^2}{dt}= 2<A(t)U(t,s)x, U(t,s)x>=-2E(t,U(t,s)x,U(t,s)x)<  0$$

%By integrating between $s$ and $t$:

%$$\int_{s}^{t}\frac{d\|U(\sigma,s)x\|^2}{d\sigma }\leq 0$$
%hence
%$$\|U(t,s)x\|^2-\|U(s,s)x\|^2\leq 0$$

%therefore

%$$\|U(t,s)x\|^2\leq \|x\|^2

%\end{proof}

We now consider the abstract Cauchy problem
\begin{equation}\label{cauchyproblem for U}
\begin{cases}
\frac{\partial u(t)}{\partial t}=A(t)u(t)\quad\text{for $t\in(0,T]$},\\
u(0)=u_0,
\end{cases}
\end{equation}
with $u_0\in L^2(\Omega,m)$.

\begin{theo} For every $u_0\in L^2(\Omega,m)$ there exists a unique $u\in C([0,T]; L^2(\Omega,m))\cap C^1((0,T]; L^2(\Omega,m))$, such that $A(t)u\in C((0,T];L^2(\Omega,m))$ and
$$\|u(t)\|_{L^2(\Omega,m)}+t\left\|\frac{\partial u(t)}{\partial t}\right\|_{L^2(\Omega,m)}+t\|A(t)u(t)\|_{L^2(\Omega,m)}\leq C\|u_0\|_{L^2(\Omega,m)},$$
for $0<t\leq T$, where $C$ is a positive constant.
% depending on the contant   \eqref{estimate AU}.
%vedi 3.64 in Yagi
Moreover, one has $u(t)=U(t,0)u_0$.
\end{theo}

\noindent For the proof, see Theorem 3.9 in \cite{Yagi}.

\section{Ultracontractivity property}\label{sec4}
\setcounter{equation}{0}

In this section we investigate the regularity of the evolution family $U(t,\tau)$. We set $\Xi:=\{(t,\tau)\in(0,T)^2\,:\,\tau<t\}$.

We recall some useful definitions, specialized to our setting.
\begin{definition}\label{proprieta U} An evolution family $\{U(t,\tau)\}_{(t,\tau)\in\Xi}$ on $L^2(\Omega,m)$ is
\begin{itemize}
	\item[1)] \emph{positive preserving} if for every $0\leq u\in L^2(\Omega,m)$ one has $U(t,\tau)u\geq 0$ for every $(t,\tau)\in\Xi$;
	\item[2)] \emph{$L^p$-contractive}, for $1\leq p\leq+\infty$, if $U(t,\tau)$ maps the set $\{u\in L^2(\Omega,m)\cap L^p(\Omega,m)\,:\,\|u\|_{L^p(\Omega,m)}\leq 1\}$ into itself for every $(t,\tau)\in\Xi$;
	\item[3)] \emph{completely contractive} if it is both $L^1$-contractive and $L^\infty$-contractive;
	\item[4)] \emph{sub-Markovian} if it is positive preserving and $L^\infty$-contractive;
	\item[5)] \emph{Markovian} if it is sub-Markovian and $\|U(t,\tau)\|_{\Lm(L^\infty(\Omega,m))}=1$.
\end{itemize}
\end{definition}

\noindent We point out that, since for every $t\in[0,T]$ the energy form $E[t,u]$ is Markovian by Theorem \ref{dirform}, the associated evolution family $U(t,\tau)$ is Markovian. In particular, this implies that the evolution family $U(t,\tau)$ is positive preserving and $L^\infty$-contractive.

Moreover, since for every $t\in[0,T]$ the bilinear form $E(t,u,v)$ is symmetric, it follows that $U(t,\tau)$ is also $L^1$-contractive, hence the evolution family is completely contractive. Therefore, by the Riesz-Thorin theorem, $\{U(t,\tau)\}_{(t,\tau)\in\Xi}$ is $L^p$-contractive for every $p\in[1,+\infty]$ and the following result holds.

%Following Theorems 5.1 and 5.2 in \cite{Daners} and taking into account the contraction property of $U(t,\tau)$, we prove ultracontractivity properties for $p\geq 1$. We start by analyzing the case $p\geq 2$.

\begin{theo}\label{theorem3}
For every $p\in [1,+\infty]$ there exists an operator $U_p(t,\tau)\in\Lm(L^p(\Omega,m))$ such that
$$U_p(t,\tau)u_0=U(t,\tau)u_0\quad\text{for every }(t,\tau)\in\Xi\,,\,\text{for every }u_0\in  L^p(\Omega,m)\cap L^2(\Omega,m).$$
Moreover, for every $\tau\geq 0$ the map $U_p(\cdot,\tau)$ is strongly continuous from $(\tau,\infty)$ to $\Lm(L^p(\Omega,m))$ for every $t\geq\tau$ and
\begin{equation}\label{Stima Contrazione U}
\|U_p(t,\tau)\|_{\Lm(L^p(\Omega,m))}\leq 1\quad\text{for every }p\geq 1.
\end{equation}
\end{theo}

%RICONTROLLARE ENUNCIATO E LEGARE BENE LA NOTAZIONE DI JEE CON QUELLA DI DELIO REMARK 3.3!!!

\bigskip

\noindent We now prove the ultracontractivity of the evolution family $U(t,\tau)$.

\begin{theo}\label{ultracontr} The evolution operator $U(t,\tau)$ is ultracontractive, i.e., for every $f\in L^1(\Omega,m)$ and $(t,\tau)\in\Xi$,
\begin{equation}\label{stima ultracontr}
\|U_1(t,\tau)f(\tau)\|_{L^\infty(\Omega,m)}\leq\left(\frac{\lambda\bar{C}}{2\beta}\right)^\frac{\lambda}{2} (t-\tau)^{-\frac{\lambda}{2}}\|f(\tau)\|_{L^1(\Omega,m)},
\end{equation}
where we recall that $\lambda=\frac{2d}{d-N+2s}$, $\bar{C}$ is the positive constant depending on $N$, $s$, $d$ and $\Omega$ appearing in \eqref{Nashineq} and $\beta>0$ is the coercivity constant of $E$.
\end{theo}

%ANCHE QUI RICONTROLLARE ENUNCIATO!

\begin{proof}
We adapt to our setting the proof of \cite[Theorem 4.3]{mugnolo}, see also \cite[Proposition 3.8]{arendtelst}.

Let $f\in H^s(\Omega)$ and let $\tau\in [0,T)$ be fixed. From \cite[Proposition III.1.2]{showalter} it holds that
\begin{equation}\notag
\frac{\de}{\de t}\|F(\cdot)\|_{L^2(\Omega,m)}^2=2\left(\frac{\de F(\cdot)}{\de t},F(\cdot)\right)_{L^2(\Omega,m)}\quad\text{for every }F\in H^s(\Omega).
\end{equation}

\noindent We remark that, if $f\in H^s(\Omega)$, then $f\in L^1(\Omega,m)$. Hence, for every $f\in H^s(\Omega)$ and a.e. $(t,\tau)\in\Xi$, from Theorem \ref{theorem1}, \eqref{corr} and the coercivity of $E[t,u]$ we have that
\begin{equation}\notag
\begin{split}
\frac{\partial}{\partial t}\|U(t,\tau)f\|_{L^2(\Omega,m)}^2&=2\left(\frac{\partial U(t,\tau)f}{\partial t},U(t,\tau)f\right)_{L^2(\Omega,m)}=2\left(A(t)U(t,\tau)f,U(t,\tau)f\right)_{L^2(\Omega,m)}\\[2mm]
&=-2E[t,U(t,\tau)f]\leq -2\beta\|U(t,\tau)f\|^2_{H^s(\Omega)},
\end{split}
\end{equation}
where $\beta$ is the (positive) coercivity constant of $E$. Then, from Nash inequality \eqref{Nashineq}, recalling that $\lambda=\frac{2d}{d-N+2s}$, it follows that
\begin{equation}\label{ultracont1}
\frac{\partial}{\partial t}\|U(t,\tau)f\|_{L^2(\Omega,m)}^2\leq -\frac{2\beta}{\bar{C}}\|U(t,\tau)f\|_{L^2(\Omega,m)}^{2+\frac{4}{\lambda}}\|U(t,\tau)f\|_{L^1(\Omega,m)}^{-\frac{4}{\lambda}},
\end{equation}
where $\bar{C}$ is the positive constant in \eqref{Nashineq} depending on $N$, $s$, $d$ and $\Omega$. Therefore, since $U(t,\tau)$ is completely contractive, from \eqref{ultracont1} we have
\begin{equation}\label{ultracont2}
\begin{split}
&\frac{\partial}{\partial t}\left(\|U(t,\tau)f\|_{L^2(\Omega,m)}^2\right)^{-\frac{2}{\lambda}}=-\frac{2}{\lambda}\|U(t,\tau)f\|_{L^2(\Omega,m)}^{-2-\frac{4}{\lambda}}\frac{\partial}{\partial t}\|U(t,\tau)f\|_{L^2(\Omega,m)}^2\\[2mm]
&\geq\frac{4\beta}{\lambda\bar{C}}\|U(t,\tau)f\|_{L^1(\Omega,m)}^{-\frac{4}{\lambda}}\geq\frac{4\beta}{\lambda\bar{C}}\|f\|_{L^1(\Omega,m)}^{-\frac{4}{\lambda}}.
\end{split}
\end{equation}
Then, integrating \eqref{ultracont2} between $\tau$ and $t$, we get 
\begin{equation}\notag
\|U(t,\tau)f\|_{L^2(\Omega,m)}^{-\frac{4}{\lambda}}\geq\frac{4\beta}{\lambda\bar{C}}\|f\|_{L^1(\Omega,m)}^{-\frac{4}{\lambda}}(t-\tau),
\end{equation}
which in turn implies that
\begin{equation}\label{ultracont L1-L2}
\|U(t,\tau)\|_{\Lm(L^1(\Omega,m)\to L^2(\Omega,m))}\leq\left(\frac{\lambda\bar{C}}{4\beta}\right)^\frac{\lambda}{4}(t-\tau)^{-\frac{\lambda}{4}}.
\end{equation}
In order to complete the proof, we need to prove an analogous bound by considering $U(t,\tau)$ as an operator from $L^2(\Omega,m)$ to $L^\infty(\Omega,m)$. We point out that, since $E(t,v,u)=E(t,u,v)$ for every $t\in [0,T]$ and $u,v\in H^s(\Omega)$, the evolution operators associated with the two forms coincide with $U(t,\tau)$. Then, from (2.22) in \cite{Daners}, we have that for $p\geq 1$ the adjoint operator $(U_p(t,\tau))^\prime$ is equal to $U_{p^\prime}(T-\tau,T-t)$ for every $0\leq\tau\leq t\leq T$.

%For $p\geq 2$ we define
%$$U_p(t,\tau)= \left(U_{p'}(T-\tau,T-t)\right)^\prime,$$
%and we note that $U_2(t,\tau)=U(t,\tau)$ on $L^2(\Omega,m)$, that is $U_p(t,\tau)$ is an extension of $U(t,\tau)$. Moreover, since the norm of the adjoint operators in Banach spaces is the same, $U_p(t,s)$ is a contraction from $L^p(\Omega,m)$ to $L^p(\Omega,m)$ also for $p\in [1,2).$ (TOGLIERE? ABBIAMO SCRITTO SOPRA GIA' CHE $U(t,\tau)$ E' UNA CONTRAZIONE PER OGNI $p\geq 1$!!!)

\noindent We now compute
\begin{equation}\label{ultracont L2-Linf}
\begin{split}
&\|U_2(t,\tau)\|_{\Lm(L^2(\Omega,m)\to L^\infty(\Omega,m))}=\|(U_2(T-\tau,T-t))^\prime\|_{\Lm(L^2(\Omega,m)\to L^\infty(\Omega,m))}\\
&=\|U_2(T-\tau,T-t)\|_{\Lm(L^1(\Omega,m)\to L^2(\Omega,m))}=\|U(T-\tau,T-t)\|_{\Lm(L^1(\Omega,m)\to L^2(\Omega,m))}\\
&\leq\left(\frac{\lambda\bar{C}}{4\beta}\right)^\frac{\lambda}{4}(t-\tau)^{-\frac{\lambda}{4}},
\end{split}
\end{equation}
where the last inequality follows from \eqref{ultracont L1-L2}.

Now, from 2) in Theorem \ref{theorem1}, combining \eqref{ultracont L1-L2} and \eqref{ultracont L2-Linf}, it finally holds that
\begin{equation}\label{ultracont L1-Linf}
\begin{split}
&\|U_1(t,\tau)\|_{\Lm(L^1(\Omega,m)\to L^\infty(\Omega,m))}=\left\|U_1\left(t,\frac{t+\tau}{2}\right)U_1\left(\frac{t+\tau}{2},\tau\right)\right\|_{\Lm(L^1(\Omega,m)\to L^\infty(\Omega,m))}\\[2mm]
&\leq\left\|U\left(t,\frac{t+\tau}{2}\right)\right\|_{\Lm(L^2(\Omega,m)\to L^\infty(\Omega,m))}\left\|U_1\left(\frac{t+\tau}{2},\tau\right)\right\|_{\Lm(L^1(\Omega,m)\to L^2(\Omega,m))}\\[2mm]
&\leq\left(\frac{\lambda\bar{C}}{2\beta}\right)^\frac{\lambda}{2}(t-\tau)^{-\frac{\lambda}{2}}.
\end{split}
\end{equation}

\end{proof}

%RIVEDERE NOTAZIONI E VARIE COSETTE NELLA DIMOSTRAZIONE (AD ESEMPIO LA PRIMA DISUGUAGLIANZA IN \eqref{ultracont L1-Linf} E LA TESI, CI VA $U_1$ O $U$???)

\bigskip

\begin{theo}\label{theorem2}
Under the hypotheses of Theorem \ref{theorem1}, the evolution operator $U(t,\tau)$ associated with the family $A(t)$ satisfies the following properties,
\begin{enumerate}
  \item[1)] for every $\theta$ such that $0\leq\theta<\eta+\frac{1}{2}$ and $0\leq\tau<t\leq T$,
	$$\mathcal{R}(U(t,\tau))\subset D(A(t)^\theta);$$
  \item[2)] for $0\leq \tau<t\leq T$, $$\left\|A(t)^\theta U(t,\tau)\right\|_{\Lm(L^{2}(\Omega,m))}\leq C_\theta (t-\tau)^{-\theta};$$
  \item[3)] for $0<\xi<\gamma<\eta+\frac{1}{2}$,
	$$\left\|A(t)^{\gamma} U(t,\tau) A(\tau)^{-\xi}\right\|_{\Lm(L^{2}(\Omega,m))} \leq C_\gamma (t-\tau)^{\xi-\gamma};$$
	\item[4)] for $\tau>0$ and $0<\xi<1$, $$\left\|[U(t+\tau,t)-U(t,t)] A(t)^{-\xi}\right\|_{\Lm(L^{2}(\Omega,m))}\leq C\tau^\xi.$$
\end{enumerate}
\end{theo}

\begin{proof} For the proof of properties 1) to 3) we refer to Section 8.1 in \cite{Yagi}. We prove 4). From Theorem \ref{theorem1}, for every $\epsilon>0$ we have
\begin{equation}\notag
\begin{split}
&\left\|[U(t+\tau,t)-U(t+\epsilon,t)] A(t)^{-\xi}\right\|_{\Lm(L^{2}(\Omega,m))}=\left\|\,\,\int_{t+\epsilon}^{t+\tau} \frac{\partial U(\sigma,t)}{\partial\sigma} A(t)^{-\xi}\,\de\sigma\right\|_{\Lm(L^{2}(\Omega,m))}\\[2mm]
&=\left\|\,\,\int_{t+\epsilon}^{t+\tau} A(\sigma)U(\sigma,t)A(t)^{-\xi}\,\de\sigma\right\|_{\Lm(L^{2}(\Omega,m))}
  \leq\int_{t+\epsilon}^{t+\tau}\left\|A(\sigma)U(\sigma,t)A(t)^{-\xi}\right\|_{\Lm(L^{2}(\Omega,m))}\,\de\sigma\\[2mm]
	&\leq C\int_{t+\epsilon}^{t+\tau}|\sigma-t|^{\xi-1}\,\de\sigma=\frac{C}{\xi}(\tau^\xi-\epsilon^\xi),
\end{split}
\end{equation}
where the last inequality follows by property 3). The conclusion then follows by passing to the limit as $\epsilon\to 0^+$ and taking into account that $U(t,\tau)$ is strongly continuous.
\end{proof}

\begin{remark}
The above properties still hold for the family of evolution operators extended to $L^p(\Omega,m)$.
\end{remark}
%%%%%%%%%%%%%%%%%%%%FIN QUI%%%%%%%%%%%%%%%%%%%%%%%

\section{The semilinear problem}\label{sec5}
\setcounter{equation}{0}

We recall the properties of the abstract inhomogeneous Cauchy problem
\begin{equation}\label{nonhomogeneous cauchyproblem for U}
\begin{cases}
\frac{\partial u(t)}{\partial t}=A(t)u(t)+f(t)\quad\text{for $t\in(0,T]$},\\
u(0)=\phi,
\end{cases}
\end{equation}
where $A(t)$ satisfies Theorem \ref{propertiesAstructutural}, $\phi\in L^2(\Omega,m)$ and  $f\in C^{0,\vartheta}([0,T],L^2(\Omega,m))$.

\begin{theo}\label{TheoremesistenzayagiF}
For every $\phi\in L^2(\Omega,m)$ and $f\in C^{0,\vartheta}([0,T], L^2(\Omega,m))$ there exists a unique $u(t)\in C([0,T]; L^2(\Omega,m))\cap C^1((0,T]; L^2(\Omega,m))$, with $A(t)u\in C((0,T]; L^2(\Omega,m))$, which satisfies \eqref{nonhomogeneous cauchyproblem for U}. Moreover, for $0<t\leq T$, one has
$$\|u(t)\|_{L^2(\Omega,m)}+t\left\|\frac{\partial u(t)}{\partial t}\right\|_{L^2(\Omega,m)} + t \|A(t)u(t)\|_{L^2(\Omega,m)}\leq C(\|\phi\|_{L^2(\Omega,m)}+\|f\|_{C^{0,\vartheta}([0,T],L^2(\Omega,m))}),$$
where $C$ is a positive constant depending on the constant in \eqref{estimate AU}.
%vedi 3.64 in Yagi
Finally,  $$u(t)=U(t,0)\phi+ \int_{0}^{t}U(t,\sigma)f(\sigma)\,\de\sigma.$$
\end{theo}

\noindent For the proof, see Theorem 3.9 in \cite{Yagi}.

\subsection{Local existence}

We now consider the abstract semilinear Cauchy problem
\begin{equation}\label{eq:5.1}
(P)\begin{cases}
\frac{\partial u(t)}{\partial t}=A(t)u(t)+J(u(t))\quad\text{for $t\in[0,T]$},\\
u(0)=\phi,
\end{cases}
\end{equation}
where $A(t)\colon D(A(t))\subset L^2(\Omega,m)\to L^2(\Omega,m)$ is the family of operators associated to the energy form $E[t,u]$ introduced in \eqref{frattale} and $\phi$ is a given function in $L^2(\Omega,m)$. We assume that for every $t\in [0,T]$ $J$ is a mapping from $L^{2p}(\Omega,m)$ to $L^2(\Omega,m)$ for $p>1$ which is locally Lipschitz, i.e., it is Lipschitz on bounded sets in $L^{2p}(\Omega,m)$,
\begin{equation}\label{LIPJ}
\|J(u)-J(v)\|_{L^2(\Omega,m)}\leq\sl{l(r)}\|u-v\|_{L^{2p}(\Omega,m)}
\end{equation}
whenever $\|u\|_{L^{2p}(\Omega,m)}\leq r,\|v\|_{L^{2p}(\Omega,m)}\leq r$, where $\sl{l(r)}$ denotes the Lipschitz constant of $J$.
We also assume that $J(0)=0$. This assumption is not necessary in all that follows, but it simplifies the calculations (see \cite{weiss1}).

In order to prove the local existence theorem, we make the following assumption on the growth of $\sl{l(r)}$ when $r\to+\infty$,
\begin{equation}\label{crescita l}
\mbox{Let}\; a:=\frac{\lambda}{4}\left(1-\frac{1}{p}\right); \quad \mbox{there exists}\,\,0<b<a\,:\,\sl{l(r)}=
{\mathcal{O}}(r^\frac{1-a}{b}),\,r\to+\infty,
\end{equation}
where $\lambda$ is defined in Proposition \ref{Nash}. We note that $0<a<1$ for $N-2s\leq\frac{d}{2}$ and $p>1$. %(ATTENZIONE QUA!!!!)\\

Let $p>1$. Following the approach in Theorem 2 in \cite{weiss1} and adapting the proof of Theorem 5.1 in \cite{La-Ve3}, we have the following result.

\begin{theorem}\label{theoesloc}
Let  condition \eqref{crescita l} hold. Let $\kappa>0$ be sufficiently small, $\phi\in L^2(\Omega,m)$ and
\begin{equation}\label{cnidato}
\limsup_{t\to 0^+}\|t^b U(t,0)\phi\|_{L^{2p}(\Omega,m)}<\kappa.
\end{equation}
Then there exists a $\overline{T}>0$ and a unique mild solution
\begin{equation}\notag
u\in C([0,\overline{T}],L^2(\Omega,m))\cap C((0,\overline{T}],L^{2p}(\Omega,m)), 
\end{equation}
with $u(0)=\phi$ and $\|t^b u(t)\|_{L^{2p}(\Omega,m)}<2\kappa$, satisfying, for every $t\in [0,\overline{T}]$,
\begin{equation}\label{rappint1}
u(t)= U(t,0) \phi +\int_0^t U(t,\tau) J(u(\tau))\,\de\tau,
\end{equation}
with the integral being both an $L^2$-valued and an $L^{2p}$-valued Bochner integral.
\end{theorem}

\begin{proof}
The proof is based on a contraction mapping argument on suitable spaces of continuous functions with values in Banach spaces. We adapt the proof of Theorem 5.1 in \cite{La-Ve3} to this functional setting; for the reader's convenience, we sketch it.\\
Let $Y$ be the complete metric space defined by
\begin{equation}\label{spazioY}
\begin{split}
Y=&\left\{u\in C([0,\overline{T}],L^2(\Omega,m))\cap C((0,\overline{T}], L^{2p}(\Omega,m))\,:\,u(0)=\phi,\right.\\[2mm]
&\left.\|t^b u(t)\|_{L^{2p}(\Omega,m)}<2\kappa \mbox{ for every} \;t\in [0,\overline{T}]\right\},
\end{split}
\end{equation}
equipped with the metric $$d(u,v)=\max\left\{\|u-v\|_{C([0,\overline{T}],L^2(\Omega,m))},\,\sup_{(0,\overline{T}]}t^b \|u(t)-v(t)\|_{L^{2p}(\Omega,m)}\right\}.$$
For $w\in Y$, let $\mathcal{F}w(t)=U(t,0)\phi +\int_0^t U(t,\tau) J(w(\tau))\,\de\tau$. Then obviously $\mathcal{F}w(0)=\phi$ and, by using arguments similar to those used in \cite[proof of Lemma 2.1]{weiss3}, we can prove that, for $w\in Y$, $\mathcal{F}w\in C([0,\overline{T}],L^2(\Omega,m))\cap C((0,\overline{T}],L^{2p}(\Omega,m))$. Then, by proceeding as in the proof of Theorem 5.1 of \cite{La-Ve3}, we prove that
\begin{equation}\label{tbFu}
\limsup_{t\rightarrow 0^+}\|t^b\mathcal{F}w(t)\|_{L^{2p}(\Omega,m)}<2\kappa \;\mbox{ for every}\; t\in [0,\overline{T}].
\end{equation}
Hence, $\mathcal{F}\colon Y\to Y$ and, by choosing suitably $\overline{T}$ and $\kappa$, we prove that it is a strict contraction.
\end{proof}
%%%%%%%%%%%%%%%%%%%%%%%%%%%%%%&&&&&&&&&& FIN QUI CON GIOIA%%%%%%%%%%%%%%%%%%%%%%%%%%%%%%%%%%

\begin{remark}\label{remteoesloc}
If $J(u)= |u|^{p-1} u$, then $\sl{l(r)}= {\mathcal{O}}(r^{p-1})$ when $r\rightarrow+\infty$. Thus condition \eqref{crescita l} is satisfied for $b=\frac{1}{p-1}-\frac{\lambda}{4p}$ with $p>1+\frac{4}{\lambda}$.
\end{remark}

We recall that, from Theorem \ref{theorem2}, we have that $\mathcal{R}(U(t,\tau))\subset D(A(t))$ for every $0<\tau\leq t$ and we can prove that the following regularity result holds (see also \cite[Theorem 5.3]{La-Ve3}).

\begin{theorem}\label{theoregloc}
Let the assumptions of Theorem \ref{theoesloc} hold.
\begin{itemize}
\item[a)] Let also condition \eqref{crescita l} hold. Then, the solution $u(t)$ can be continuously extended to a maximal interval $(0,T_\phi)$  as a solution of \eqref{rappint1}, until $\|u(t)\|_{L^{2p}(\Omega,m)}<\infty$;
\item[b)] one has that $$u\in C([0,T_\phi),L^2(\Omega,m))\cap  C((0,T_\phi), L^{2p}(\Omega,m))\cap C^1((0,T_\phi),L^2(\Omega,m)),$$
$$ Au(t)\in\;C((0,T_\phi);\;L^2(\Omega,m))$$
and $u$ satisfies
\begin{equation*}
\frac{\partial u(t)}{\partial t}=A(t)u(t)+J(u(t)) \quad\mbox{for every}\;\; t \in (0,T_\phi)
\end{equation*}
and $u(0)=\phi$ (that is, $u$ is a classical solution).
\end{itemize}
\end{theorem}

\begin{proof}
For the proof of condition a), we follow \cite[Theorem 2]{weiss1}. From the proof of Theorem \ref{theoesloc}, it turns out that the minimum existence time for the solution to the integral equation is as long as $\|t^b U(t,\tau)\phi\|_{L^{2p}(\Omega,m)}\leq\kappa$ (see also Corollary 2.1. in \cite{weiss1}).

\noindent To prove that the mild solution is classical, we use the classical regularity results for linear equations (see Theorem \ref{TheoremesistenzayagiF}) by proving that $J(u)\in C^{0,\vartheta}((0,T],L^2(\Omega,m))$ for any fixed $T<T_\phi$.

Taking into account the local Lipschitz continuity of $J(u)$, it is enough to show that $u(t)$ is H\"older continuous from $(\epsilon,
T)$ into $L^{2p}(\Omega,m)$ for every $\epsilon>0$. Let $\psi= u(\epsilon)$ and we set $w(t)= U(t,0)\psi +\int_0^t U(t,\tau) J(w(\tau))\,\de\tau$. If we prove that
$$w(t)\in C^0([0,T];\;L^{2p}(\Omega,m))\cap C^1([0,T]), L^2(\Omega,m))$$ and
$$A(t)w\;\in C([0,T];\;L^2(\Omega,m)),$$ then, as $u(t+\epsilon)= w(t)$ due to the uniqueness of the solution of \eqref{rappint1}, we deduce that $$u(t) \in C^1([\epsilon,T+\epsilon);\;L^2(\Omega,m))\cap C([\epsilon,T+\epsilon), L^{2p}(\Omega,m))$$ and
$$A(t)u(t)\in\;C([\epsilon,T+\epsilon);\;L^2(\Omega,m))$$ for every $\epsilon>0$, hence $u(t)$ is a classical solution (see claim b)).\\
Let $\sup_{t\in (0,T)}\|w\|_{L^{2p}(\Omega,m)}\leq r$. Since $U(t,0)$ is differentiable in $(\epsilon, T)$, then it is H\"older continuous for any exponent $\gamma\in (0,1)$.
We now prove that $$v(t)=\int_0^t U(t,\tau) J(w(\tau))\,\de\tau$$ is H\"older continuous too.
Let $0\leq t\leq t+\sigma\leq T$; then
\begin{equation}\notag
\begin{split}
&v(t+\sigma)-v(t)=\int_{0}^{t+\sigma}U(t+\sigma,\tau)J(w(\tau))\,\de\tau-\int_{0}^{t} U(t,\tau)J(w(\tau))\,\de\tau\\[2mm]
&=\int_{0}^{t}(U(t+\sigma,\tau)-U(t,\tau))J(w(\tau))\,\de\tau+\int_{t}^{t+\sigma}U(t+\sigma,\tau)J(w(\tau))\,\de\tau=:v_1(t)+v_2(t).
\end{split}
\end{equation}

For the function $v_1$, for $0<\gamma<1$ it holds that
\begin{equation}\notag
\begin{split}
&\|v_1(t)\|_{L^{2p}(\Omega,m)}\leq\int_{0}^{t}\|(U(t+\sigma,t)U(t,\tau)-U(t,\tau)) J(w(\tau))\|_{L^{2p}(\Omega,m)}\,\de\tau\\[2mm]
&=\int_{0}^{t}\|(U(t+\sigma,t)-\textrm{Id})A(t)^{-\gamma}A(t)^{\gamma}U(t,\tau) J(w(\tau))\|_{L^{2p}(\Omega,m)}\,\de\tau\\[2mm]
&=\int_{0}^{t}\left\|\left(\int_{t}^{t+\sigma}A(\xi)U(\xi,t)A(t)^{-\gamma}\,\de\xi\right) A(t)^{\gamma}U(t,\tau) J(w(\tau))\right\|_{L^{2p}(\Omega,m)}\,\de\tau\\[2mm]
&\leq\int_{0}^{t}\left(\int_{t}^{t+\sigma}\left\|A(\xi)U(\xi,t)A(t)^{-\gamma}\right\|_{\Lm(L^{2p}(\Omega,m))}\,\de\xi\right)\cdot\\[2mm]
&\cdot\left\|A(t)^{\gamma}U\left(t,\frac{\tau+t}{2}\right)U\left(\frac{\tau+t}{2},\tau\right)J(w(\tau))\right\|_{L^{2p}(\Omega,m)}\,\de\tau.
\end{split}
\end{equation}

From the ultracontractivity of $U(t,\tau)$ and the Riesz-Thorin theorem, one has
\begin{equation}\label{RT}
\|U(t,\tau)\|_{\Lm(L^2(\Omega,m)\to L^{2p}(\Omega,m))}\leq C\left((t-\tau)^{-\frac{\lambda}{4}}\right)^{1-\frac{1}{p}},
\end{equation}
where we recall that $\lambda=\frac{2d}{d-N+2s}$ and $C$ is a positive constant depending on $N$, $s$, $d$, $p$ and $\Omega$.\\
Then, taking into account \eqref{RT} and parts 2) and 3) of Theorem \ref{theorem2}, we have
\begin{equation}\notag
\begin{split}
&\|v_1(t)\|_{L^{2p}(\Omega,m)}\leq C\int_{0}^{t}\left(\int_{t}^{t+\sigma}|\xi-t|^{\gamma-1}\,\de\xi\right)\left\|A^\gamma(t)U\left(t,\frac{\tau+t}{2}\right)\right\|_{\Lm(L^{2p}(\Omega,m))}\cdot\\[2mm]
&\cdot\left\|U\left(\frac{\tau+t}{2},\tau\right)J(w(\tau))\right\|_{L^{2p}(\Omega,m)}\de\tau\\[2mm]
&\leq \tilde{C}\int_{0}^{t} \frac{\sigma^\gamma}{\gamma}\left(\frac{t-\tau}{2}\right)^{-\gamma} \left(\frac{t-\tau}{2}\right)^{-\frac{\lambda}{4}\left(1-\frac{1}{p}\right)} \|J(w(\tau))\|_{L^2(\Omega,m)}\,\de\tau\\[2mm]
&\leq \tilde{C}\int_{0}^{t} \frac{\sigma^\gamma}{\gamma} \left(\frac{t-\tau}{2}\right)^{-\gamma}\left(\frac{t-\tau}{2}\right)^{-a} \sl{l(r)} r\,\de\tau,
\end{split}
\end{equation}
where $\tilde{C}$ is a positive constant depending on the constant in \eqref{RT} and $\gamma$. If we choose $\gamma<1-a$, we obtain $\|v_1(t)\|_{L^{2p}(\Omega,m)}\leq C\sigma^\gamma$, for a suitable positive constant $C$ depending also on $T$ and $r$.

As to the function $v_2$, using again \eqref{RT} we have
\begin{equation}\notag
\begin{split}
&\|v_2(t)\|_{L^{2p}(\Omega,m)}\leq \int_t^{t+\sigma}\|U(t+\sigma,\tau)J(w(\tau))\|_{L^{2p}(\Omega,m)}\,\de\tau\\[2mm]
&=\int_{t}^{t+\sigma}\|U(t+\sigma,\tau)\|_{\Lm(L^2(\Omega,m)\to L^{2p}(\Omega,m))}\|J(w(\tau))\|_{L^2(\Omega,m)}\,\de\tau\leq \tilde{C}\frac{\sigma^{1-a}}{1-a} \sl{l(r)} r \leq C\sigma^{1-a},
\end{split}
\end{equation}
for a suitable positive constant $C$ depending on the constant in \eqref{RT}, $r$ and $a$. Therefore, if $\gamma<1-a$, $v(t)$ is H\"older continuous on $[0,T]$ with exponent $\gamma$.
\end{proof}

%RIVEDERE LA DIMOSTRAZIONE! GLI ESPONENTI E TUTTI I DETTAGLI! IN PARTICOLARE L'ULTIMA CREDO CHE SIA UNA DISUGUAGLIANZA E NON UN'UGUAGLIANZA

\subsection{Global existence}

We now give a sufficient condition on the initial datum for obtaining a global solution, by adapting Theorem 3 (b) in \cite{weiss2}.

\begin{theorem}\label{theoesglob}
Let condition \eqref{crescita l} hold. Let $q:=\frac{2\lambda p}{\lambda+4pb}$, $\phi\in L^q(\Omega,m)$ and $\|\phi\|_{L^q(\Omega,m)}$ be sufficiently small. Then there exists $u\in C([0,\infty), L^q(\Omega,m))$ which is a global solution of \eqref{rappint1}.
\end{theorem}

\begin{proof}
Since $q<2p$, as in \eqref{RT} from the ultracontractivity of $U(t,\tau)$ and the Riesz-Thorin theorem it follows that $U(t,\tau)$ is a bounded operator from $L^q(\Omega,m)$ to $L^{2p}(\Omega,m)$ with
$$\|U(t,\tau)\|_{\Lm(L^q(\Omega,m)\to L^{2p}(\Omega,m))}\leq M (t-\tau)^{-\frac{\lambda}{2}\left(\frac{1}{q}-\frac{1}{2p}\right)}\equiv M (t-\tau)^{-b},$$ where $M$ is a positive constant depending $N$, $s$, $d$, $p$, $b$ and $\Omega$. Hence, we have that
$$ \|t^b U(t,0)\phi\|_{L^{2p}(\Omega,m)}\leq M \|\phi\|_{L^q(\Omega,m)};$$ by choosing $\|\phi\|_{L^q(\Omega,m)}$ sufficiently small, from Theorem \ref{theoesloc} we have that there exists a local solution of \eqref{rappint1} $u \in C([0,T], L^q(\Omega,m))$. Furthermore, from Theorem  \ref{theoesloc} we also have that $u \in C((0,T],L^{2p}(\Omega,m))$ and $ \|t^b u(t) \|_{L^{2p}(\Omega,m)}\leq 2M\|\phi\|_{L^q(\Omega,m)}$.

From Theorem \ref{theoregloc} a), if we prove that $\|u(t)\|_{L^{2p}(\Omega,m)}$ is bounded for every $t>0$, then $u(t)$ is a global solution. We will prove that $\|t^b u(t)\|_{L^{2p}(\Omega,m)}$ is bounded for every $t>0$, and we will use the notations of the proof of Theorem \ref{theoesloc}.\\
We choose $\Lambda>0$ such that $l(r)\leq\Lambda r^{\frac{1-a}{b}}$ for $r\geq 1$. Then
\begin{equation}\notag
\begin{split}
&\|t^b u(t)\|_{L^{2p}(\Omega,m)}\leq M\|\phi\|_{L^q(\Omega,m)}+
 t^b \int_0^t \|U(t,\tau)\|_{\Lm(L^2(\Omega,m)\to L^{2p}(\Omega,m))}\|J(u(\tau))\|_{L^2(\Omega,m)}\,\de\tau\\[2mm]
&\leq M\|\phi\|_{L^q(\Omega,m)}+ M\Lambda\left(2M \|\phi\|_{L^q(\Omega,m)}\right)^{\frac{1-a}{b}} t^b \int_0^t (t-\tau)^{-a} \tau^{a-1-b}\|\tau^b u(\tau)\|_{L^{2p}(\Omega,m)}\,\de\tau\\[2mm]
&\leq M\|\phi\|_{L^q(\Omega,m)} +M\Lambda\left(2M \|\phi\|_{L^q(\Omega,m)}\right)^{\frac{1-a}{b}}\sup_{t\in [0,T]} \|t^b u(t)\|_{L^{2p}(\Omega,m)} \int_0^1 (1-\tau)^{-a} \tau^{a-1-b}\,\de\tau.
\end{split}
\end{equation}
We point out that the integral on the right-hand side of the above inequality is finite.
Let now $f(T)=\sup_{t\in [0,T]} \|t^b u(t)\|_{L^{2p}(\Omega,m)}$. Then $f(T)$ is a continuous nondecreasing function with $f(0)=0$ which satisfies
$$f(T)\leq M\|\phi\|_{L^q(\Omega,m)}+\left(2M\|\phi\|_{L^q(\Omega,m)}\right)^{\frac{1-a}{b}} \Lambda BM f(T),$$
where $B:=\int_0^1 (1-\tau)^{-a} \tau^{a-1-b}\,\de\tau>0$. If $M\|\phi\|_{L^q(\Omega,m)}\leq \epsilon$ and $2^{\frac{1-a+b}{b}}\Lambda BM\epsilon^{\frac{1-a}{b}}<1$, then $f(T)$ can never be equal to $2\epsilon$. If it could, we would have $2\epsilon\leq\epsilon+(2\epsilon)^{\frac{1-a+b}{b}}\Lambda BM$, i.e., $\epsilon\leq(2\epsilon)^{\frac{1-a+b}{b}} \Lambda BM$, which is false if $\epsilon>0$ is small enough.

This proves that, for $\|\phi\|_{L^q(\Omega,m)}$ sufficiently small, $\|t^b u(t)\|_{L^{2p}(\Omega,m)}$ must remain bounded and the claim follows.
\end{proof}

%RIVEDERE ALCUNI DETTAGLI DELLA DIMOSTRAZIONE! ALCUNI CONTI NON MI SONO CHIARI...

\subsection{The strong formulation}

We now give a strong formulation of the abstract Cauchy problem $(P)$ in \eqref{eq:5.1}.

\begin{theorem}\label{esistfrattale} Let $\alpha$ be as defined in \eqref{definizione alpha} and $s\in (0,1)$ be such that $N-d<2s<N$. Let $u$ be the unique solution of problem $(P)$. Then for every fixed $t\in (0,T]$, one has
\begin{equation}\label{pbforte}
\begin{cases}
\frac{\partial u}{\partial t}(t,x)+\B u(t,x)=J(u(t,x)) &\text{for a.e. $x\in\Omega$,}\\[2mm]
\frac{\partial u}{\partial t}+C_s\Ncal u+bu+\Theta^t_\alpha (u)=J(u)\quad &\text{in $(B^{2,2}_\alpha(\partial\Omega))'$},\\[2mm]
u(0,x)=\phi(x) &\text{in $L^2(\Omega,m)$}.
\end{cases}
\end{equation}
\end{theorem}

\begin{proof} For every $t\in (0,T]$, we multiply the first equation of problem $(P)$ by a test function $\varphi\in\D(\Omega)$ and then we integrate on $\Omega$. Then from \eqref{corr} we obtain
\begin{equation}\notag
\begin{split}
\int_\Omega \frac{\partial u}{\partial t}(t,x)\,\varphi(x)\,\de\La_N&=\int_{\Omega} A(t)u(t,x)\,\varphi(x)\,\de\La_N+\int_{\Omega} J(u(t,x))\,\varphi(x)\,\de\La_N\\
&=-E(t,u,\varphi)+\int_{\Omega} J(u(t,x))\,\varphi(x)\,\de\La_N.
\end{split}
\end{equation}
Since $\varphi$ has compact support in $\Omega$, after integrating by parts we get
\begin{equation}\label{forte}
\frac{\partial u}{\partial t}+\B u=J(u)\quad\text{in $(\D(\Omega))'$}.
\end{equation}
By density, equation \eqref{forte} holds in $L^2(\Omega)$, so it holds for a.e. $x\in\Omega$. We remark that, since $J(u(t,\cdot))\in L^2(\Omega,m)$, it also follows that, for each fixed $t\in (0,T]$, $u\in V(\B,\Omega)$. Hence, we can apply Green formula \eqref{fracgreen}.

We now take the scalar product in $L^2(\Omega,m)$ between the first equation of problem $(P)$ and $\varphi\in H^s(\Omega)$. Hence we get
\begin{equation}\label{formvariaz}
\left(\frac{\partial u}{\partial t},\varphi\right)_{L^2(\Omega,m)}=(A(t)u,\varphi)_{L^2(\Omega,m)}+(J(u),\varphi)_{L^2(\Omega,m)}.
\end{equation}
By using again \eqref{corr}, we have that
\begin{equation}\notag
\begin{split}
&\int_\Omega \frac{\partial u}{\partial t}(t,x)\,\varphi\,\de\La_N+\int_{\partial\Omega} \frac{\partial u}{\partial t}(t,x)\,\varphi(x)\,\de\mu\\[2mm]
&=-\frac{C_{N,s}}{2}\iint_{\Omega\times\Omega} K(t,x,y)\frac{(u(t,x)-u(t,y))(\varphi(x)-\varphi(y))}{|x-y|^{N+2s}}\,\de\La_N(x)\de\La_N(y)
\\[2mm]
&-\int_{\partial\Omega} b(t,x)\,u(t,x)\,\varphi(x)\,\de\mu-\langle\Theta^t_\alpha(u),\varphi\rangle+\int_{\Omega}J(u(t,x))\,\varphi(x)\,\de\La_N+\int_{\partial\Omega} J(u(t,x))\,\varphi(x)\,\de\mu.
\end{split}
\end{equation}
Using \eqref{fracgreen} and \eqref{forte}, we obtain for every $\varphi\in H^s(\Omega)$ and for each $t\in (0,T]$
\begin{equation}\label{bordo}
\begin{split}
\int_{\partial\Omega} \frac{\partial u}{\partial t}(t,x)\,\varphi(x)\,\de\mu=&-\left\langle C_s\Ncal u,\varphi\right\rangle-\int_{\partial\Omega} b(t,x)\,u(t,x)\,\varphi(x)\,\de\mu\\[2mm]
&-\langle\Theta^t_\alpha(u),\varphi\rangle+\int_{\partial\Omega} J(u(t,x))\,\varphi(x)\,\de\mu.
\end{split}
\end{equation}
Hence the boundary condition holds in $(B^{2,2}_\alpha(\partial\Omega))'$.
\end{proof}

\vspace{1cm}

\noindent {\bf Acknowledgements.} The authors have been supported by the Gruppo Nazionale per l'Analisi Matematica, la Probabilit\`a e le loro
Applicazioni (GNAMPA) of the Istituto Nazionale di Alta Matematica (INdAM). The authors report there are no competing interests to declare.

\end{document}